\documentclass[12pt]{article}

\usepackage{amsfonts}
\usepackage{amsmath, amsthm, upref}
\usepackage{enumerate}
\usepackage{color}
\RequirePackage[colorlinks, urlcolor = blue, citecolor = blue, urlcolor = blue, linkcolor = blue]{hyperref}

\usepackage{multimedia}
\usepackage{graphics}
\usepackage{epsfig}
\usepackage[greek,english]{babel}
\usepackage{textcomp}
\usepackage{eurosym}
\usepackage[mathscr]{euscript}

\newcommand{\mbf}{\mathbb{F}}
\newcommand{\mbg}{\mathbb{G}}
\newcommand{\mby}{\mathbb{Y}}

\newtheorem{theorem}{Theorem}
\newtheorem{assumption}[theorem]{Assumption}

\newtheorem{corollary}{Corollary}

\newtheorem{definition}{Definition}
\newtheorem{example}{Example}

\newtheorem{lemma}{Lemma}
\newtheorem{proposition}{Proposition}
\newtheorem{remark}[theorem]{Remark}
\newcommand{\bee}{\begin{equation}}
\newcommand{\eee}{\end{equation}}

\begin{document}

\date{\today }

\title{Expansion of a Filtration with a Stochastic Process: The Information Drift}

\author{L\'eo Neufcourt\thanks{Department of Statistics, Michigan State University, East Lansing, Michigan. email: leo.neufcourt@gmail.com} and Philip Protter\thanks{Statistics Department, Columbia University, New York, NY 10027; supported
in part by NSF grant DMS-1612758. email: pep2117@columbia.edu.}}
\maketitle
\begin{abstract}
When expanding a filtration with a stochastic process it is easily possible for a semimartingale no longer to remain a semimartingale in the enlarged filtration. Y. Kchia and P. Protter indicated a way to avoid this pitfall in 2015, but they were unable to give the semimartingale decomposition in the enlarged filtration except for special cases. We provide a way to compute such a decomposition, and moreover we provide a sufficient condition for It\^o processes to remain It\^o processes in the enlarged filtration. This has significance in applications to Mathematical Finance. 

\end{abstract}

\section*{Introduction}

The expansion of filtrations and its consequences on stochastic processes have been a subject of interest since the early papers of It\^o~\cite{Ito} and M. Barlow~\cite{Barlow}, both of which appeared in 1978. In the past there have been two basic techniques. One is known as initial enlargement and the other is known as progressive enlargement.  These techniques are well known and we do not review them here. The interested reader can consult the recent book (2017) of A. Aksamit and M. Jeanblanc~\cite{AJ}.

A third type of filtration enlargement was recently proposed by Y. Kchia and P. Protter~\cite{KP}, published in 2015. This approach involves the continuous expansion of a filtration via information arriving from a continuous time process. It is a complicated procedure, and the authors were able to show when the semimartingale property was preserved, but in general they were not able to describe how its decomposition into a local martingale and a finite variation process occurred, except for some special cases. In particular, as explicit examples showed, even if the original semimartingale was an It\^o process, the expanded process, even when still a semimartingale, need no longer be an It\^o process. 
This has significant consequences in the theory of mathematical finance, for example. 

In this article we give a sufficient condition such that not only does the semimartingale remain a semimartingale under the filtration expansion, but it also remains an It\^o process (ie, the sum of a stochastic integral with respect to a Brownian motion and a finite variation process whose paths are absolutely continuous a.s.). If we write the It\^o process in the form
\bee\label{e1}
X_t=X_0+\int_0^tH_sdB_s+\int_0^t\alpha_sds
\eee
then we refer to the integrand $(\alpha_s)_{s\geq 0}$ as the \emph{information drift}. 

In Section~\ref{motivation} we provide examples of possible applications to three different areas coming from operations research, physics, and mathematical finance. In Section~\ref{sec-exp-fil-id} we give the technical necessities, and define what we mean by \emph{the information drift} of a semimartingale. One of the tools that underlies our approach is that of the weak convergence of $\sigma$ algebras and (especially) filtrations; this is treated in Section~\ref{sec-cv-fil}. In Section~\ref{sec-exp-process} we treat the main subject of this article, that of the expansion of a filtration via the dynamic inclusion of the information of another stochastic process. This section contains what we believe is this paper's main contribution, namely Theorem~\ref{thm-Tprocess}.

\section{Motivation}\label{motivation}

It has been argued that expansions of a filtration
with a general stochastic process 
are not really more appropriate than successive initial expansions
 to describe informational asymmetry in concrete applications.
While the discrete nature of real-world data is undeniable, 
e.g. in the context of high-frequency trading,
the continuous limit is nevertheless often the key to understand 
discrete high-frequency behaviors.

\subsection{Truck Freight Routing}

Our first potential application comes from the modeling of trucking freight. One of the largest companies in the U.S. for freight moved by truck is the Yellow Freight Company, known as YRC Freight. Originally it began in 1924 in Oklahoma City, Oklahoma, but has long since become a national company. A map of its distribution routings (as a snapshot in time) is in Figure 1. The basic problems related to a national truck routing framework can be summarized as follows~\cite{Huseyin}:

\begin{figure}[h]
\begin{center}
\includegraphics[scale=0.65]{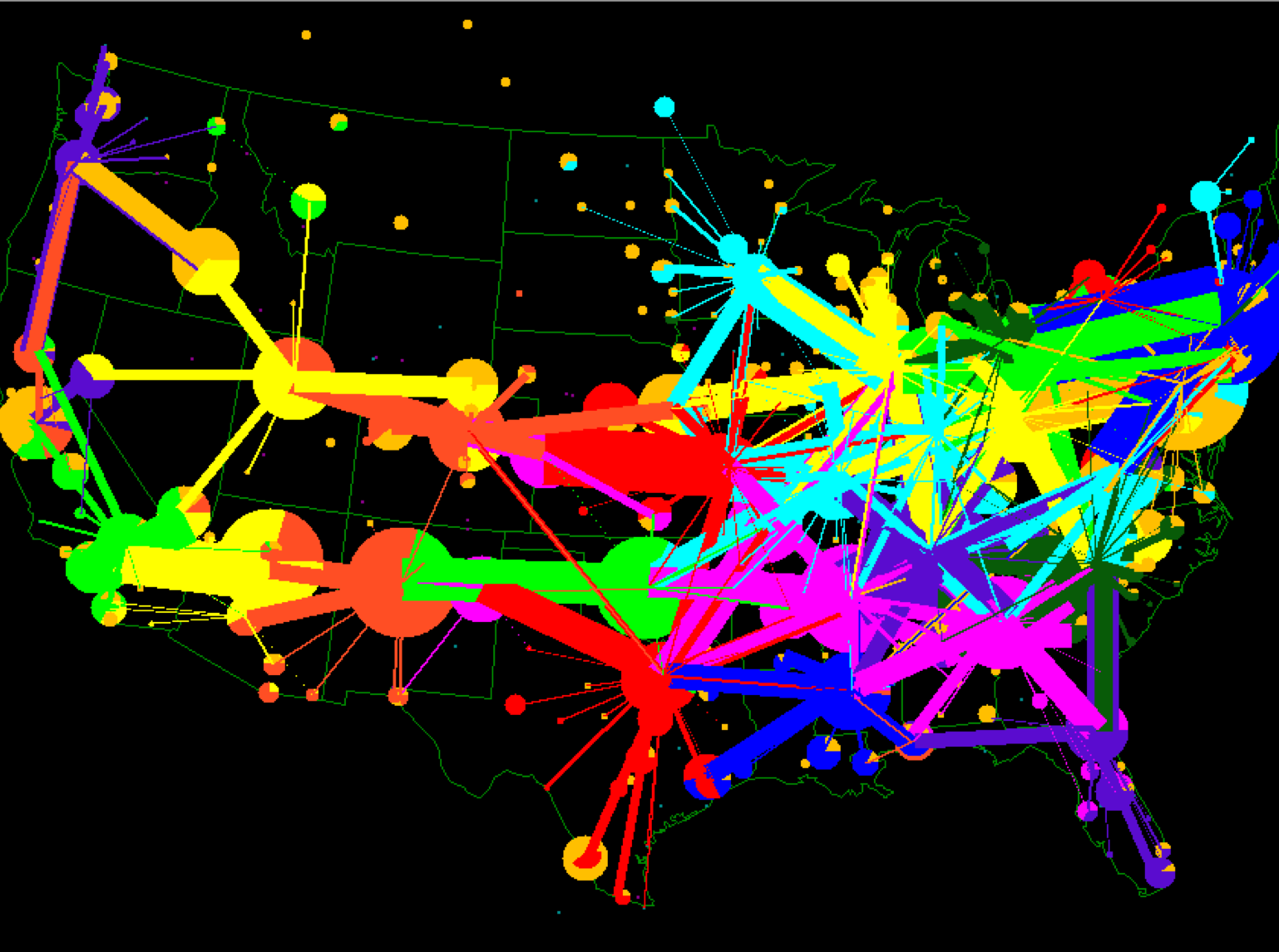}
\end{center}
\caption{Yellow Freight Route Map, courtesy of H. Topaloglu}
\label{Yellow}
\end{figure}

\begin{itemize}
\item We have a set of indivisible and reusable resources
\item Tasks arrive over time according to a probability distribution
\item At each decision epoch, we have to decide what tasks to cover
\item Resources generally have complex characteristics determining the feasibility of assigning a resource to a task
\item We are interested in maximizing efficiency over a finite time horizon
\end{itemize}

When making resource allocation decisions, one wants first to secure a ``good" first period profit and also make sure that the allocation of trucks is going to be favorable in the second time period. Using dynamic programming, one can construct a function that gives the expected worth of a truck at a certain location at a certain time period. See for example the expository article of Powell and Topaloglu~\cite{PT}. As the trucks move, however, information is constantly received by management via GPS Tracking as well as communication from both the drivers and customers. It is reasonable to model this information flow as continuous. If the underlying filtration is generated by truck positions and customer orders according to each ``decision epoch,'' then this extra information, corrupted perhaps by noise, can be modeled as a continuous enlargement of the underlying filtration. 

The goal is to maximize the sum of the immediate profits, as well as the expected worth of a truck in the next time period. In this way, the functions estimate the impact of a decision in the current time period on the future time periods. Once the system is translated into profits in this way, we can construct a mathematical model along the lines of
\bee\label{e1}
X_t=M_t+A_t\text{ is a semimartingale modeling the profits at time }t\geq 0
\eee
where the 	``profits" include both profits and losses. A loss can occur from penalties for a load being delivered late, or from - for example - a road accident delaying the load, and injuring the truck and/or (more importantly) the driver. Accidents are frequent, as truck drivers are pushed to their endurance limits, and can fall asleep at the wheel while driving. The local martingale $M$ could be, for example, noise in the system, due (for example) to faulty reporting, failures of the GPS system, or sabotage, as well as load theft. Given  Central Limit Theorem arguments, it is reasonable to model such noise as Brownian motion.

The finite variation term of~\eqref{e1} can be modeled as $\int_0^th_sds$, denoting the cumulative profits (or losses) from time $0$ to time $t>0$. If one accepts this model, then we are in a framework where we can consider filtration enlargement along the lines of this paper. In this case the filtration enlargement process would come from information entering into the system as the semimartingale $X$ of~\eqref{e1} evolves with time. Such information can take the form of order updates or cancelations, unforeseen supply chain issues, truck breakdowns and unreliable drivers, and the like. 

\subsection{High Energy Particle Collisions}

We believe these mathematical models can be applied to help to understand the behavior of heavy ion collisions. Michigan State University is building a new Facility for Rare Isotope Beams. High energy particles will be destroyed, allowing physicists to follow the reactions, by observing the light fragments ($\pi$, p, n, d, He) that are emitted. One then hopes to infer the detailed structure of the nuclei. This is related to nuclear liquid-gas phase transition. 

These collisions typically behave as follows, according to the theory presented in the highly regarded paper~\cite{SG}: When two nuclei collide at high energy, a strong non-linear shock wave is formed. High density, pressure and temperature are formed in the region of the collision. The system expands due to its large internal pressure until the collisions between the particles cease to occur. Then the hydrodynamic description loses its validity and this is when one sees the light fragments. The light fragments are the data, and the rest is (accepted, but) speculative theory. Fragment formation is a topic of ``great current interest" in connection with the discovery of the nuclear liquid-gas phase transition. 

Were we to model the behavior of the particles as systems of stochastic differential equations (for example), then we would want to enlarge the underlying filtration with the evolving entropy and the viscous effects, which the physical theories show present a serious effect on the evolution of the particles. This would represent our filtration enlargement and subsequent change in the evolution of the particles. 

\subsection{Insider Trading}

A model of a specialized kind of insider trading is the original motivation for our study. Note that the legality of insider trading is a complicated subject filled with nuance. There are types of trades which are (at least currently) perfectly legal but are most accurately described as insider trading. A dramatic example is provided by the high frequency traders known as co-locators. Our analysis is inspired by the work of~\cite{KP}. These companies place their trading machines next to the computers that process the trades of the stock exchange. They rent these co-locations directly from the stock exchanges themselves, and the fiber optics are carefully measured so that no one co-locator has a physical advantage over another. Trades take place at intervals of 0.007 seconds. These co-locators use various specialized orders (such as ``immediate or cancel" orders) effectively to ``see" the limit order book in the immediate future, and thus determine if a given stock is likely to go up or to go down in the immediate future. That this can be done has been recently described and illustrated in preliminary work of Neufcourt, Protter, and Wang~\cite{NPW}.  The techniques of this paper can be used to model this very phenomenon. A beginning of how to do this was suggested in~\cite{KP}. Here we develop the necessary mathematics to carry this program forward. 

We could model the stock price by any of a variety of standard models (for example, stochastic volatility models within a Heston paradigm), but include extra information from the structure of the limit order book, which evolves through time, available only to the co-locators. The larger filtration would then affect the semimartingale decompositions of the Heston models, giving a new model for the insiders. This would mean the insiders would have a new collection of risk neutral measures, \emph{a priori} different from those of the market in general. This would allow them to determine if financial derivatives (for example call options) are fairly priced or not. They could then leverage their inside information given this knowledge. 

\subsection*{Acknowledgement}
The authors wish to thanks Professor Monique Jeanblanc for several helpful comments upon her reading of a preliminary draft of this paper.  The authors also wish to thank H. Topaloglu for allowing us to use his diagram of the Yellow Freight Trucking Routes within the U.S.

\section{Expansions of filtrations, semimartingales and information drift} \label{sec-exp-fil-id}

\subsection{Preliminaries} \label{sub-prelim}

We suppose given a complete probability space $(\Omega,\mathcal{A},\mathbb{P})$, 
a time interval $I\subset[0,\infty)$ containing zero, a filtration $\mbf:=(\mathcal{F}_t)_{t\in I}$ 
satisfying the usual hypotheses, and we assume implicitly that random variables and stochastic processes take values
in a Polish space $(E,\mathcal{E})$. We consider a continuous  $\mbf$-semimartingale $S$, 
i.e. a stochastic process admitting a decomposition 

\begin{equation*}
S=S_0+M+A \label{eq-SMD}
\end{equation*}

\noindent
where $M$ is a continuous $\mbf$-local martingale and $A$ a continuous $\mbf$-adapted process of finite variation on compact intervals (\emph{finite variation process}), both starting at zero%
\footnote{We recall that semimartingales can equivalently be defined as the class of ``natural'' stochastic integrators 
\emph{via} the Bitcheler-Dellacherie characterization, and include most processes commonly considered 
such as solutions of Markov or non Markov SDEs (see \cite{PPbook}).}.
That is, we assume $M_0=A_0=0$ a.s.

\begin{assumption}\label{a1}
All semimartingales in this article are assumed to be continuous unless otherwise explicitly noted.
\end{assumption}

\begin{definition}\label{d1} An expansion (or augmentation) of the filtration $\mbf$ is another filtration $\mbg$ such that $\mathcal{F}_t\subset\mathcal{G}_t$ for every $t\in I$.
\end{definition}

Given an $\mbf$-semimartingale $S$ with decomposition 
$S=S_0+M+A$
and any expansion $\mbg$ of $\mbf$,
 it is straightforward that $A$ is also a $\mbg$-finite variation process.
However $M$ need not in general be a $\mbg$-local martingale (see It{\^o}'s founding example \cite{Ito}), 
nor even a $\mbg$-semimartingale 
(one can consider for instance a Brownian motion and the constant filtration generated by its whole path). 

We will restrict our focus, therefore, to the properties of a $\mbf$-local martingale $M$ with respect to an expansion $\mbg$.
When $M$ is a $\mbg$-semimartingale we wish to know its decomposition into the sum of a $\mbg$ local martingale $\widetilde{M}$ and a $\mbg$ adapted finite variation process $\widetilde{A}$ 
(see \cite[Chapter VI,Theorem 4]{PPbook}). 
Let us state in passing if $M$ is an $\mbf$-Brownian motion then 
$\widetilde{M}$ is a continuous $\mbg$-local martingale
with quadratic variation $[\widetilde{M}, \widetilde{M}]_t=t$ hence it's a $\mbg$ Brownian motion by the L\'evy characterization.


In order to perform tractable computations it is often desirable to consider expansion models where
$M$ is a $\mbg$-semimartingale, and has a $\mbg$-decomposition of the form 
$$M=\widetilde{M}+\int_0^.\alpha_sd[M,M]_s,$$
for some $\mbg$-adapted measurable process $\alpha$. 
Such a decomposition is unique due to the assumption of the continuity of the paths. 
Proposition 1.5 in \cite{ADI} establishes that the process $\alpha$ can be taken to be $\mathcal{G}$-predictable,
in which case uniqueness actually occurs up to indistinguishibility, 
which justifies our introduction of $\alpha$ in the next section as 
\emph{the information drift} of the expansion $\mbg$. \\

\noindent
Given a filtration $\mby$ under which $M$ is a semimartingale we will denote 
$\mathscr{S}(\mby, M)$ the set of $\mby$-predictable and $M$-integrable
stochastic processes, as well as $\mathscr{H}^1(\mby, M)$ and $\mathscr{H}^2(\mby, M)$ the 
subsets containing the processes $H$ satisfying respectively $\mathbb{E}\int_I|\alpha_s|d[M,M]_s<\infty$ and 
$\mathbb{E}\int_I\alpha_s^2d[M,M]_s < \infty$. When there is no ambiguity we will omit the dependency on $M$ 
in these notations.\\

We begin with a result that is doubtless well known to experts but perhaps is prudent to include for the sake of completeness. 

\begin{proposition} \label{prop-usual-hyp} 
Let $\check{\mby}$ be a filtration,
{$\mathcal{N}$ the set of 
$\mathbb{P}$-null sets in $\mathcal{A}$,}
and $\mby$ the right-continuous filtration defined by 
$\mby_t := \bigcap_{u>t}(\check{\mby}_u\vee\mathcal{N})$
\begin{enumerate}[1.]
\item Every $\check{\mby}$-semimartingale $S$ with decomposition $S=S_0+M+A$ is also a 
${\mby}$-semimartingale with the same decomposition $S_0 + M+A$.
\item Every continuous and $\check{\mby}$-adapted ${\mby}$-semimartingale with decomposition $S_0+M+A$ is 
also a $\check{\mby}$-semimartingale with decomposition 
$S_0+M+A$ (for suitable modifications of $M$ and $A$).
\end{enumerate}
\end{proposition}
\begin{proof} 
Since $\mathcal{N}$ is independent from any other $\sigma$-algebra, (1.) is a simple application of
Theorem 2 in  \cite{PPbook}.
The proof of (2.) follows from similar arguments: let $S$ be a $\mby$-semimartingale 
with decomposition $M+A$; since $M$ is continuous, $M_s = \lim_{n\to\infty} M_{s-1/n}$ so that $M_s$ 
is measurable with respect to $\bigvee_{n\geq 1} \mby_{s-1/n}\subset\check{\mby}_s\vee\mathcal{N}$, 
which means that $M$ is $\check{\mby}\vee\mathcal{N}$-adapted and thus admits a modification which is $\check{\mby}$-adapted. 
If $S$ is $\check{\mby}$-adapted, 
so is $A$ and $S=S_0+M+A$ is consequently a $\check{\mby}$-semimartingale decomposition of $S$. 
\end{proof}



\subsection{Expansions of filtrations as a weak change of probability}

Relationships between expansions of filtration, existence of the information drift,
existence of a local martingale measure and absence of arbitrage have been noticed in numerous papers. Some of the deepest results are contained in \cite{Ank-embeddings,Ank-thesis}. The closest approach to ours is \cite{Kar-NA1-enlargement} 
but it is limited to the particular cases of initial and progressive expansions. 

In this section we calibrate classical arguments to our focus on the information drift and our general expansion setup.
The general idea is that an expansion of filtration for which there exists an information 
drift can be seen as a \emph{bona fide} change of probability as long as the information drift has sufficient integrability. This idea dates back at least to the seminal paper of Imkeller~\cite{Imkeller}. Even in the absence of required integrability the simple existence of the information drift still allows one to define a ``weak version'' of a change of probability.

This has implications in finance where the axiomatic absence of arbitrage
is closely related to the existence of a \emph{risk-neutral} probability under which the price process 
$S\in L^2(\Omega \times I;\mathbb{R}^d)$ 
is a local martingale, which can then be used 
to price derivatives according to the value of their hedging portfolio.
These two properties are actually equivalent when absence of arbitrage is 
understood in the sense of No Free Lunch with Vanishing Risk (see \cite{DS-NFLVR}).

In an asymmetric information setting, one typically considers a price process which 
is a local martingale under a risk neutral probability for the initial filtration $\mathcal{F}$.
In this paradigm absence of arbitrage for the agent with the additional information contained in $\mathcal{G}$ 
means precisely the existence of a risk neutral probability with respect to $\mathcal{G}$, in which case 
the effect of the expansion is contained in the change of the risk neutral probability.
Relaxing the integrability required to define a \emph{bona fide} change of probability leads to the slightly weaker 
notion of arbitrage known as No Arbitrage of the First Kind ($\text{NA}_1$) or No Unbounded Profit with Bounded Risk (NUPBR)
around which a variation of the martingale pricing theory has been developed \cite{Kar-NA1-ftap,Kar-NA1-loc-mart-num,Kar-NA1-market,Kar-time,Kar-NA1-enlargement}.


The next theorem introduces the \textit{information drift} $\alpha$ 
as a logarithmic derivative of the change of risk neutral probability,
which characterizes its existence as a no-arbitrage condition 
and is enlightening in regard to its nature.
It generalizes results well known in the case of initial and progressive expansions (e.g. \cite{Kar-NA1-enlargement,AJ}).

\begin{theorem} \label{thm-alpha-lm-deflator}
Let $M$ be a not necessarily continuous $\mathcal{F}$-local martingale 
and $\mathcal{G}$ an expansion of $\mathcal{F}$. The following assertions are equivalent:
\begin{enumerate}[(i)]
\item \label{item-lmd} There exists a positive $\mathcal{G}$-local martingale $Z$ such that $ZM$ is a $\mathcal{G}$-local martingale.
\item \label{item-id} There exists a $\mathcal{G}$-predictable process $\alpha\in\mathcal{S}(\mathcal{G})$ such that $M$ is a $\mathcal{G}$-semimartingale with decomposition 
$$M=\widetilde{M}+\int_0^.\alpha_sd[M,M]_s.$$
\end{enumerate}
When $(i)$ and $(ii)$ are satisfied it additionally holds that:
\begin{enumerate}[(a)]
\item  $\alpha$ is unique $d\mathbb{P}\times d[M,M]$-a.s.
\item $d[Z,M]_t=\alpha_t d[M,M]_t,d\mathbb{P}\times d[M,M]$-a.s.
\item $Z$ is of the form $Z=\mathcal{E}(\alpha\cdot (M-\alpha\cdot [M,M])+L)=\mathcal{E}(\alpha\cdot (M-\alpha\cdot [M,M]))\mathcal{E}(L)$, where $L$ is a $\mathcal{G}$-local martingale with $[L,M]=0,d\mathbb{P}\times d[M,M]$-a.s.
\end{enumerate}
\end{theorem}

Note that $Z$ as in (\ref{item-lmd}) is called a local martingale deflator (for $M$), and that (\ref{item-id}) 
can be equivalently formulated as $M-\int_0^.\alpha_sd[M,M]_s$ being a $\mathcal{G}$-local martingale.

\begin{definition}[Information drift\footnote{This most appropriate expression can already be found in 
\cite{BS} and \cite{Shannon}.}]\label{def-id}
The $\mathcal{G}$-predictable process $\alpha\in\mathcal{S}(\mathcal{G})$ defined by Theorem \ref{thm-alpha-lm-deflator}, i.e. such that $M$ is a $\mathcal{G}$-semimartingale with decomposition 
$$M=\widetilde{M}+\int_0^.\alpha_sd[M,M]_s,$$
is called the \textbf{information drift} of $M$ (between the filtrations $\mathcal{F}$ and $\mathcal{G}$). \\
In such situation we also say that $\mathcal{G}$ admits an information drift (for $M$).
\end{definition}

The core of the proof of Theorem \ref{thm-alpha-lm-deflator} is an application of It{\^o}'s formula.

\begin{lemma} \label{lemma-alpha-lm-deflator} Let $Z$ be a positive local martingale and let $M$ be a general {semimartingale}. Then $M^Z:=ZM$ is a local martingale if and only if $M-\int_0^.\frac{1}{Z_s}d[Z,M]_s$ is a local martingale.
\end{lemma}
\begin{proof}
First note that, as $Z$ is a positive semimartingale, $M^Z$ is a semimartingale if and only $M$ is 
(this follows for instance from It{\^o}'s formula).
It{\^o}'s formula also gives $d(\frac{1}{Z})_t=\frac{-1}{Z_t^2}dZ_t+\frac{1}{Z_t^3}d[Z,Z]_t$. Hence
$$dM_t=M^Z_t d(\frac{1}{Z_t})+\frac{1}{Z_t}dM^Z_t+d[M^Z,\frac{1}{Z}]_t
= \frac{-M^Z_t}{Z_t^2}dZ_t+\frac{1}{Z_t}dM^Z_t+dA_t,$$
where $dA_t := \frac{M^Z_t}{Z_t^3}d[Z,Z]_t-\frac{1}{Z_t^2}d[M^Z,Z]_t$.
Since $\int \frac{-M^Z_s}{Z_s^2}dZ_s$ defines a local martingale, it follows that $M^Z$ is a local martingale if and only if $M-A$ is. 
Finally the identity $d[Z,ZM]=Md[Z,Z]+Zd[Z,M]$ leads to 
$A=-\int_0^.\frac{1}{Z_s}d[Z,M]_s$.
\end{proof}

\begin{proof}[Proof of Theorem \ref{thm-alpha-lm-deflator}]
Suppose first that $Z$ is a positive $\mathcal{G}$-local martingale such that $ZM$ is a $\mathcal{G}$-local martingale. It follows from the lemma that $M-\int_0^.\frac{1}{Z_t}d[Z,M]_s$ is a $\mathcal{G}$-local martingale. The Kunita-Watanabe inequality implies that there exists a predictable c\`{a}dl\`{a}g process $\beta$ such that $d[Z,M]=\beta d[M,M]$, 
and $\alpha:=\frac{\beta}{Z}$ is such that 
$d[Z,M]=\alpha Z d[M,M]$ and $M-\int_0^.\alpha_sd[M,M]_s$ 
is a $\mathcal{G}$-local martingale. 
Moreover by the predictable representation theorem (see \cite[Chapter IV]{PPbook}) there exist a $\mathcal{G}$-predictable process 
$J\in\mathcal{S}$ and a $\mathcal{G}$-local martingale $L$ 
such that $[L,M-\alpha\cdot [M,M]]=[L,M]=0$ and 
$Z_t=Z_0+\int_0^t J_sZ_sd(M_s-\alpha\cdot [M,M])+\int_0^tZ_sdL_s$. It follows from $[Z,M] = [(JZ)\cdot M+Z\cdot L,M] = (JZ)\cdot [M,M]+Z\cdot [L,M]=(JZ)\cdot [M,M]$ that $J=\alpha,d\mathbb{P}\times dt$-a.s. and that $Z=\mathcal{E}(\alpha\cdot (M-\alpha\cdot [M,M])+L)=\mathcal{E}(\alpha\cdot (M-\alpha\cdot [M,M]))\mathcal{E}(L)$.

Now suppose conversely that $M-\int_0^.\alpha_sd[M,M]_s$ is a $\mathcal{G}$-local martingale. 
The positive $\mathcal{G}$-local martingale 
$Z^\alpha:=\mathcal{E}(\alpha\cdot M)$ satisfies 
$dZ^\alpha_t=\alpha Z^\alpha_t dM_t$ 
and $[Z^\alpha,M]=[(\alpha Z^\alpha)\cdot M,M]=(\alpha Z^\alpha)\cdot [M,M]$. Hence $M-\int_0^.\frac{1}{Z_{s^-}^\alpha}d[Z^\alpha,M]_s$ is an $\mathcal{F}$-local martingale, and $Z^\alpha M$ is a $\mathcal{G}$-local martingale by Lemma \ref{lemma-alpha-lm-deflator}.
\end{proof}

The widespread use of risk neutral pricing justifies \emph{per se} an attention to the particular case where the local martingale deflator defined in Theorem \ref{thm-alpha-lm-deflator} defines an equivalent measure known as an Equivalent Local Martingale Measure (ELMM) or risk neutral probability, which validates the No Free Lunch with Vanishing Risk variation of the axiom of absence of arbitrage.
Note that as a Theorem \ref{thm-alpha-lm-deflator}
the existence of a $\mathcal{G}$-ELMM 
implies the existence of an information drift.
Conversely we have the following result:

\begin{theorem} \label{thm-NFLVR} Let $\mathcal{G}$ be an expansion of $\mathcal{F}$ and $M$ be a $\mathcal{F}$-local martingale. If $M$ has an information drift $\alpha$ and $\mathcal{E}(\alpha\cdot (M-\alpha\cdot [M,M]))$ is a (uniformly) integrable $\mathcal{G}$ martingale, then $\mathbb{Q}:=\mathcal{E}(\alpha\cdot (M-\alpha\cdot [M,M]))\cdot\mathbb{P}$ defines a $\mathcal{G}$-ELMM for $M$.
\end{theorem}

Theorem \ref{thm-NFLVR} follows from Theorem \ref{thm-alpha-lm-deflator} together with the next lemma.

\begin{lemma} Let $M$ be a local martingale and $\alpha\in\mathcal{S}$. The following statements are equivalent:
\begin{enumerate}[(i)]
\item $\mathcal{E}(\alpha\cdot M)$ is a martingale
\item $\mathcal{E}(\alpha\cdot M)$ is a uniformly integrable martingale 
\item $\mathbb{E}\mathcal{E}(\alpha\cdot M)=1$
\end{enumerate}
\end{lemma}
\begin{proof}
A stochastic exponential $\mathcal{E}(\alpha\cdot M)$ is a local martingale, hence a supermartingale, hence
$\mathbb{E}\mathcal{E}(\alpha\cdot M)\leq 1$ and the lemma follows easily.
\end{proof}

\begin{remark} When $\alpha$ in Theorem \ref{thm-alpha-lm-deflator} exists, it can define a change of probability 
only if 
\begin{equation*}
\int_0^T\alpha_s^2d[M,M]_s <\infty \ a.s.
\end{equation*}
Breaches of this condition have been studied (cf, eg,~\cite{JP2005}) as \emph{immediate} arbitrage opportunities.
\end{remark}

\subsection{Information drift of the Brownian motion}

In the case where the $\mathcal{F}$-local martingale $M$ under consideration is a Brownian motion $W$ the information drift takes a particular expression due to the predictability of the quadratic variation, which nicely illustrates its name.

\bigskip

\begin{theorem} \label{thm-alpha-from-cond-expectation}
Let $W$ be an $\mathcal{F}$-Brownian motion.
If there exists a process $\alpha\in\mathscr{H}^1(\mathcal{G})$ such that 
$W-\int_0^.\alpha_sds$ is a $\mathcal{G}$-Brownian motion,
then we have:
\begin{enumerate}[(i)]
\item $\alpha_s=\lim\limits_{\substack{t\to s\\ t>s}}\mathbb{E}[\frac{W_t-W_s}{t-s}|\mathcal{G}_s]$
\item $\alpha_s=\frac{\partial}{\partial t}\mathbb{E}[W_t|\mathcal{G}_s]\Big|_{t=s}$.
\end{enumerate}
Conversely if there exists a process  
$\alpha\in\mathscr{H}^1(\mathcal{G})$ 
satisfying $(i)$ or $(ii)$,
then and $\alpha$ is the information drift of $W$,
i.e.
$W-\int_0^.\alpha_sds$ is a $\mathcal{G}$-Brownian motion.
\end{theorem}
\begin{proof} Since $W$ is $\mathcal{G}$-adapted it is clear that $(i)$ and $(ii)$ are equivalent. Suppose that $M-\int_0^.\alpha_udu$ is a $\mathcal{G}$-martingale. It follows that, for every $s\leq t$,
$\mathbb{E}\Big[W_t-W_s-\int_s^t\alpha_u du|{\mathcal{G}}_s\Big]
=0$,
hence
\begin{align*}
 \mathbb{E}[\int_s^t\alpha_udu|{\mathcal{G}}_s\Big]
  \ = \ \int_s^t\mathbb{E}[\alpha_u|{\mathcal{G}}_s]du.
\end{align*} 
By differentiating with respect to $t$ we obtain
$\mathbb{E}[\alpha_t|\mathcal{G}_s]=\frac{\partial}{\partial t}\mathbb{E}[W_t|\mathcal{G}_s]$
which establishes $(ii)$.
Conversely if $(ii)$ holds it is also clear that 
$$
\mathbb{E}[W_t - W_s - \int_s^t\alpha_u du | \mathcal{G}_s]
=
\mathbb{E}[W_t - W_s - 
(\mathbb{E}[W_t|\mathcal{G}_t] -  \mathbb{E}[W_s|\mathcal{G}_s])
| \mathcal{G}_s]
=
0.$$
\end{proof}

\begin{remark}
Theorem \ref{thm-alpha-from-cond-expectation} holds for any expansion $\mathcal{G}$ of the Brownian filtration 
$\mathcal{F}$, and can be naturally extended to general martingales with deterministic quadratic variation.
Similar results can be found in \cite[Exercises 28 and 30, p. 150]{PPbook} as well in \cite{JMP} with a 
focus on the integrated information drift.
\end{remark}

\begin{example}\label{cor-alpha-from-gaussian-x} 
Suppose that for every $s\leq t$ we have a representation of the form 
\begin{equation}\label{eq-xi-mu}
\mathbb{E}[W_t|\mathcal{G}_s]=\int_0^s\xi_u\mu_{s,t}(du),
\end{equation}

\noindent
where $(\mu_{s,t})_{s\leq t}$ is a family of finite signed measures adapted in the $s$ variable
to $\mathcal{G}$ and $\xi$ is a $\mathcal{G}$-measurable stochastic process which is $\mu_{s,t}$-integrable.
If the information drift $\alpha\in\mathscr{H}^1$ exists then by Theorem \ref{thm-alpha-from-cond-expectation} it is given by
$\alpha_s
=\frac{\partial}{\partial t}\mathbb{E}[W_t|\mathcal{G}_s]\Big|_{t=s}
= \frac{\partial}{\partial t}\int_0^s \xi_u \mu_{s,t}(du)$.

Let us suppose additionally that,
locally in $(s,t)$ on some ``sub-diagonal'' 
$\{(s,t): 0\vee(t-\epsilon)\leq s\leq t\leq T\}$ with $\epsilon > 0$,
the map $t\mapsto\mu_{s,t}$ is continuously differentiable and satisfies a domination
$|\partial_t\mu_{s,t}(du)|\leq C \eta_{s}(du)$ for some constant $C>0$ and finite measure $\eta_{s}$, with 
$\int |\xi_u| \eta_s(du)<\infty$. We can then apply the classical theorem for derivation under the integral sign 
to obtain
$$\alpha_s = \int_0^s \xi_u\partial_t\mu_{s,t}\Big|_{t=s}(du).$$

\noindent
Note that any measure $\mu_{s,s}$ as in (\ref{eq-xi-mu}) needs to satisfy
$\int_0^s (\xi_u-W_s) \mu_{s,s}(du)=0$ a.s.
\end{example}

\begin{example}
Suppose that the optional projection of $W$ on $\mathcal{G}$ can be decomposed on the Wiener space 
$\mathcal{C}([0,s])$ of continuous functions
as
$$\mathbb{E}[W_t|\mathcal{G}_s] 
= \int \xi_{s,t}(\omega) \mathcal{W}_s(d\omega),$$ 
where 
$\mathcal{W}$ is the Wiener measure,
for every $t\geq 0$ the process
$(s,\omega) \mapsto \xi_{s,t}(\omega)$ 
is $\mathcal{G}$-adapted,  
satisfies $\int |\xi_{s,t}| \mathcal{W}_s(d\omega) <\infty$.
Suppose additionally that for every $s\geq 0,\omega\in\mathcal{W}$,
the map $t\mapsto \xi(s,t)$ is differentiable and that its derivative 
is dominated by a non-random integrable function.
Then 
$$\alpha_s
= \frac{\partial}{\partial t} \int \xi_{s,t}(\omega) 
\mathcal{W}_s(d\omega)
 = \int \frac{\partial}{\partial t} \xi_{s,t}(\omega) 
 \mathcal{W}_s(d\omega)
.$$
\end{example}

\section{Weak convergence of filtrations in $L^p$}\label{sec-cv-fil}

In this section we introduce the convergence of filtrations in $L^p$ and establish some of its elementary properties. 
We then focus on the conservation of the semimartingales and convergence of the semimartingale 
decomposition and information drift. 
We restrict here our study to semimartingales with continuous paths, and accordingly $M$ will represent 
in this section a continuous $\mathcal{F}$-local martingale (see Section \ref{sub-prelim}). 
Thus when $M$ is a semimartingale it is always a special semimartingale so that we can always consider its canonical semimartingale decomposition,
which we recall also allows us to consider \emph{the} information drift of $M$.

\subsection{Convergence of $\sigma$-algebras and fitrations\label{sub-cv-fil}}

\noindent
Recall that the convergence of random variables in $L^p$ is defined for $p>0$ as
$$Y^n\xrightarrow[n\to\infty]{L^p}Y\iff\mathbb{E}[|Y^n-Y|^p]\xrightarrow[n\to\infty]{}0.$$

\begin{definition} 
Let $p>0$. 
\begin{enumerate}[1.]
\item We say that sequence of $\sigma$-algebras $(\mby^n)_{n\geq 1}$ converges in $L^p$ 
to a $\sigma$-algebra $\mby$, denoted $\mby^n\xrightarrow[n\to\infty]{L^p}\mby$, 
if one the following equivalent conditions is satisfied:
\begin{enumerate}[(i)]
\item $\forall B\in\mby,\mathbb{E}[1_B|\mby_n]\xrightarrow[n\to\infty]{L^p}1_B$
\item $\forall Y\in L^p(\mby, \mathbb{P}),\mathbb{E}[Y|\mby_n]\xrightarrow[n\to\infty]{L^p}Y$
\end{enumerate}
\item 
We say that a sequence of filtrations $(\mby^n)_{n\geq 1}$ converges \emph{weakly} in $L^p$ 
to a filtration $\mby$, denoted $\mby^n\xrightarrow[n\to\infty]{L^p}\mby$, 
if $\mby_t^n\xrightarrow[n\to\infty]{L^p}\mby_t$ for every $t\in I$.
\end{enumerate}
\end{definition}
Note that there is no uniqueness of the 
limits in the above definition: for instance, any sequence of $\sigma$-algebras converges in $L^p$ to 
the trivial filtration, and any $\sigma$-algebra contained in a limiting $\sigma$-algebra is also a limiting 
$\sigma$-algebra.
Before studying the behavior of semimartingales we establish elementary properties of the weak convergence of filtrations 
in $L^p$, following corresponding properties of convergence of random variables, and using similar 
techniques as in \cite{KP}.

\begin{proposition}
Let $1\leq p\leq q$. Convergence of $\sigma$-algebras in $L^q$ implies convergence in $L^p$.
\end{proposition}

\begin{remark}[Weak convergence] Weak convergence in the sense of 
\cite{CMS}, \cite{Hoover1991}, and \cite{KP} 
corresponds to requesting convergence in probability instead of $L^p$ in the previous definition.
Weak convergence is therefore weaker than convergence of $\sigma$-algebras in $L^p$ for any $p\geq 1$.
\end{remark}

\begin{proposition}\label{prop-VHn-providentiel-martingale}
For every \textbf{non-decreasing} sequence of $\sigma$-algebras $(\mby^n)_{n\geq 1}$ 
and every $p\geq 1$ we have 
$$\mby_n\xrightarrow[n\to\infty]{L^p}\mathcal{\bigvee}_{n\in\mathbb{N}}\mby_n.$$
\end{proposition}
\begin{proof}
Let $\mby:=\bigvee_{n\in\mathbb{N}}\mby_n$ and consider 
$Y\in L^p(\mby,\mathbb{P}),p\geq 1$. $(\mathbb{E}[Y|\mby^n])_{n\geq1}$
 is a closed (and uniformly integrable) martingale which converges to $\mathbb{E}[Y|\mby]$ 
 a.s. and in $L^p$.
\end{proof}

\noindent
This is also the consequence of the following more general property.

\begin{proposition} Let $\mby$ be a $\sigma$-algebra and 
$(\mby^n)_{n\geq1},(\mathcal{Z}^n)_{n\geq1}$ two sequences of $\sigma$-algebras.
\begin{center}
If $\mby^n\xrightarrow[n\to\infty]{L^2}\mby$ and $\mby^n\subset\mathcal{Z}^n$ 
then $\mathcal{Z}^n\xrightarrow[n\to\infty]{L^2}\mby$.
\end{center}
\end{proposition}
\begin{proof} Let $Y\in L^2(\mby,\mathbb{P})$. Since $\mathbb{E}[Y|\mathcal{Z}^n]$ minimizes 
$\lVert U-Y\rVert_{L^2}$ over all $\mathcal{Z}^n$-measurable random variables $U\in L^2(\mathcal{Z}^n,\mathbb{P})$, we have
$$\lVert\mathbb{E}[Y|\mby^n]-Y\lVert_{L^2}
\geq
\lVert\mathbb{E}[Y|\mathcal{Z}^n]-Y\rVert_{L^2}.$$
\end{proof}

\begin{corollary}\label{cor-doncDiscRC}
Let $\check{\mby}$ be a filtration, $(\check{\mby}^n)_{n\geq1}$ a sequence of filtrations and 
consider the sequence of right-continuous filtrations defined by 
${\mby_t^n}:=\bigcap_{u>t}\check{\mby}_u^n,t\in I$. 
\begin{center}
If $\check{\mby}^n\xrightarrow[n\to\infty]{L^2}\check{\mby}$ then 
${\mby^n}\xrightarrow[n\to\infty]{L^2}\check{\mby}$.
\end{center}
\end{corollary}

\begin{proposition}\label{prop-propConv}
Let $\mby$ be a $\sigma$-algebra, $(\mby^n)_{n\geq1}$ a sequence of $\sigma$-algebras and 
$(Z^n)_{n\geq1}$ a sequence of random variables.
\begin{center}
If $\mby^n\xrightarrow[n\to\infty]{L^1}\mby$ and $Z^n\xrightarrow[n\to\infty]{L^1} Z$,
then $\mathbb{E}[Z^n|\mby^n]\xrightarrow[n\to\infty]{L^1}\mathbb{E}[Z|\mby]$.
\end{center}
\end{proposition}
\begin{proof}
\begin{equation*}
\Big| \mathbb{E}[Z|\mby]
-\mathbb{E}[Z^n|\mby^n] \Big|
\leq 
\Big|\mathbb{E}[Z|\mby]-\mathbb{E}[Z|\mby^n]\Big|
+ \mathbb{E}[\Big| Z-Z_n\Big| |\mby^n]
\end{equation*}
The first term converges to zero because of convergence of filtrations, and the second term converges to $0$ 
in $L^1$ by the tower property of conditional expectations.
\end{proof}

\begin{proposition}\label{prop-YnvZn-to-YvZ} 
Let $\mby,\mathcal{Z},(\mby^n)_{n\geq 1},(\mathcal{Z}^n)_{n\geq1}$ be 
$\sigma$-algebras and $p\geq 1$. If $\mby^n\xrightarrow[L^p]{n\to\infty}\mby$ and 
$\mathcal{Z}^n\xrightarrow[L^p]{n\to\infty}\mathcal{Z}$, then
$$\mby^n\vee\mathcal{Z}^n\xrightarrow[L^p]{n\to\infty}\mby\vee\mathcal{Z}.$$
\end{proposition}
\begin{proof}
This is actually the core of the proof of Proposition 2.4 of \cite{CMV} 
where the conclusion is reached for weak convergence.
\end{proof}

\subsection{Stability of semimartingales \label{sub-SMG}}

\begin{theorem}[Stability of the semimartingale property]\label{thm-stabSM} \ \\ 
Let $(\mathcal{G}^n)_{n\geq1}$ be a sequence of filtrations and suppose 
that $M$ is a $\mathcal{G}^n$ semimartingale with decomposition $M=:M^n+A^n$ for every $n\geq1$.

If $\mathcal{G}^n\xrightarrow[n\to\infty]{L^2}\mathcal{G}$ and $\sup_n \mathbb{E}\int_0^T d|A_t^n|<\infty$, 
then $M$ is a $\mathcal{G}$-semimartingale.
\end{theorem}
\begin{proof}
We prove the stronger statement that $M$ is a $\mathcal{G}$-quasimartingale (see \cite{PPbook}, Chapter III).
Let $H$ be a simple $\mathcal{G}$-predictable process of the form $H_t:=\sum_{i=1}^ph_i1_{]u_i,u_{i+1}]}(t)$, 
where each $h_i$ is $\mathcal{G}_{u_i}$-measurable, $0\leq u_1 < ... < u_{p+1}=T$, and $\forall i=1...p, |h_i|\leq1$. 
We define $H_t^n:=\mathbb{E}[H_t|\mathcal{G}_t^n] = \sum_{i=1}^ph_i^n1_{]u_i,u_{i+1}]}(t)$, with 
$h_i^n:=\mathbb{E}[h_i|\mathcal{G}_{t_i}^n]$. $H^n$ is a simple $\mathcal{G}^n$-predictable process bounded by 
$1$, hence 
$$|\mathbb{E}[(H^n\cdot M)_T]|
= |\mathbb{E}\int_0^TH_t^ndA_t^n|
\leq \mathbb{E}\int_0^T|dA_t^n|
\leq\sup_n \mathbb{E}\int_0^T|dA_t^n|.$$
Now
$\mathbb{E}[((H-H^n)\cdot M)_T]
= \sum_{i=1}^p \mathbb{E}[(h_i-h_i^n)(M_{u_{i+1}}-M_{u_{i}})]$
and we obtain from the Cauchy-Schwarz inequality:
\begin{align*}
|\mathbb{E}[((H-H^n)\cdot M)_T]|
& \leq 
\sum_{i=1}^p \mathbb{E}[|h_i-h_i^n| |M_{u_{i+1}}-M_{u_{i}}|]\\
& \leq 
\sum_{i=1}^p \sqrt{\mathbb{E}[(h_i-h_i^n)^2] \mathbb{E}[(M_{u_{i+1}}-M_{u_{i}})^2]}\\
& \leq 
\sqrt{\sup_{i=1...p}\mathbb{E}[(h_i-h_i^n)^2]} \sum_{i=1}^p\sqrt{\mathbb{E}[(M_{u_{i+1}}-M_{u_{i}})^2]}.
\end{align*}

We show that the two terms on the right side converge to 0. First, the Cauchy-Schwarz inequality implies
{
$\sum_{i=1}^p\sqrt{\mathbb{E}[(M_{u_{i+1}}-M_{u_{i}})^2]}
\leq
\sqrt{p\sum_{i=1}^p\mathbb{E}[(M_{u_{i+1}}-M_{u_{i}})^2]}<\infty.$
}
Then, given the convergence of $\sigma$-algebras 
$\mathcal{G}_t^n\xrightarrow[n\to\infty]{} \mathcal{G}_t,0\leq t\leq T$ 
and the bound $\mathbb{E}[h_i^2]\leq1<\infty$ we have 
$\sup_{i=1...p}\mathbb{E}[(h_i-h_i^n)^2]\xrightarrow[n\to\infty]{} 0$.
Thus $\mathbb{E}[((H-H^n)\cdot M)_T]\xrightarrow[n\to\infty]{} 0$
so that 
$|\mathbb{E}[(H\cdot M)_T]|
\leq\sup_n\mathbb{E}\int_0^T|dA_t^n|$
which establishes that $M$ is a $\mathcal{G}$-quasimartingale by the Bichteler-Dellacherie characterization (see \cite{PPbook}).
\end{proof}

\noindent
The last theorem shows that the semimartingale property is conserved through the convergence of filtrations in $L^2$, 
as long as the sequence of finite variation terms is absolutely uniformly bounded. The next result proves that if the 
sequence of finite variation terms converges, convergence of filtrations in $L^1$ is sufficient.

\begin{theorem}[Convergence of semimartingale decomposition]\label{thm-stabDec}
Let $(\mathcal{G}^n)_{n\geq1}$ be a sequence of filtrations and suppose that $M$ is a $\mathcal{G}^n$ 
semimartingale with decomposition $M=:M^n+A^n$  for every $n\geq1$.

If $\mathcal{G}^n\xrightarrow[n\to\infty]{L^1}\mathcal{G}$ and 
$\mathbb{E}[\int_0^Td|A_t^n-A_t|]\xrightarrow[n\to\infty]{} 0$ then $M$ is a $\mathcal{G}$-semimartingale 
with decomposition $M=:\widetilde{M}+A$ (where $\widetilde{M}:=M-A$).
\end{theorem}
\begin{proof}
Since $A^n$ is $\mathcal{G}$-predictable $\forall n\geq1$ its limit $A$ is also $\mathcal{G}$-predictable. 
Moreover $\int_0^T|dA_t|\leq\int_0^T|dA_t^n-dA_t|+\int_0^T|dA_t^n|<\infty$ so $A$ has finite variations. 
We also have $\forall t\in[0,T],A_t^n\xrightarrow[n\to\infty]{L^1}A_t$ and by a localization argument we can 
assume without loss of generality that $M$ and $A$ are bounded.Hence according to Proposition \ref{prop-propConv} 
we also have for every $s\leq t$
$$0=\mathbb{E}[(M_t-A_t^n)-(M_s-A_s^n)|\mathcal{G}_s^n]\xrightarrow[n\to\infty]{L^1}\mathbb{E}[(M_t-A_t)-(M_s-A_s)|\mathcal{G}_s],$$
which shows that $M-A$ is a $\mathcal{G}$-martingale.
\end{proof}

\subsection{Convergence of the information drifts\label{sub-ConvID}}

We now study more specifically the convergence of square-integrable information drifts relying on the Hilbert structure 
of the subset $\mathcal{S}^2(M)$ containing the $\mathcal{F}$-predictable processes in 
$L^2(\Omega \times [0, T],\mathcal{A}\otimes\mathcal{B},d\mathbb{P}\times d[M,M])$,
where we will denote $\mathcal{P}$ the predictable 
$\sigma$-algebra.
We show that a uniform bound on 
the norms of the successive information drifts ensures 
the existence of a limit which is the information drift 
of the limit filtration, extending the original argument in \cite{Shannon} 
to our general convergence of filtrations 
and the setting of an expansion with a stochastic process.

\bigskip

\begin{theorem}[Convergence of information drifts] \label{thm-convID} 
Let $(\mathcal{G}^n)_{n\geq1}$ be a \textbf{non-decreasing} sequence of filtrations and suppose that $M$ is a 
$\mathcal{G}^n$-semimartingale with decomposition $M=:M^n+\int_0^.\alpha_s^nd[M,M]_s$  for every $n\geq1$ 
for some process $\alpha^n\in \mathcal{S}(\mathcal{G}^n)$. \\
If $\mathcal{G}^n\xrightarrow[n\to\infty]{L^2}\mathcal{G}$ and 
$\sup_{n\geq1}\mathbb{E}\int_0^T(\alpha^n_u)^2d[M,M]_u<\infty$  
then $M$ is a ${\mathcal{G}}$-semimartingale with decomposition 
$$M=:\widetilde{M}+\int_0^.\alpha_sd[M,M]_s,$$
where $\alpha\in\mathcal{S}^2(\mathcal{G},M)$.
\end{theorem}
\begin{proof} 
Suppose first that there exists such $\alpha$. Up to a localizing sequence, we can assume that $[M,M]$ is bounded by $C > 0$. 
With $A^n:=\int_0^.\alpha_s^n d[M,M]_s$ and $A:=\int_0^.\alpha_s d[M,M]_s$, we have 
\begin{align*}
\mathbb{E}\int_0^T d|A_t^n-A_t|
 \ = \
\mathbb{E}\int_0^T|\alpha_s^n-\alpha_s| d[M,M]_s \\
 \ \leq \ 
 \mathbb{E}
\Big[
 \sqrt{[M,M]_T}
\sqrt{\int_0^T(\alpha_s^n-\alpha_s)^2d[M,M]_s}
\Big] \\
 \ \leq \
\sqrt{C}
 \sqrt{\mathbb{E}\int_0^T(\alpha_s^n-\alpha_s)^2d[M,M]_s}
\xrightarrow[n\to\infty]{}0
\end{align*}
such that $M$ has decomposition $\widetilde{M}+\int_0^.\alpha_sd[M,M]_s$ by Theorem \ref{thm-stabDec}.

Now if $m\leq n$ we have in the Hilbert space 
$L^2(M):=L^2([0, T] \times \Omega,\mathcal{P},d[M,M]\times d\mathbb{P})$ the following orthogonality:
\begin{eqnarray*}
\mathbb{E}\int_0^T(\alpha^m-\alpha^n)\alpha^nd[M,M]_s
& = & \mathbb{E}\int_0^T\alpha^n [(dM_t-dM_t^m) - (dM_t-dM_t^n)] \\
& = & \mathbb{E}\int_0^T\alpha^n (dM_t^n-dM_t^m)\\
& = & 0
\end{eqnarray*}
because $\alpha^n$ is $\mathcal{G}^n$ and $\mathcal{G}^m$ predictable which implies that both stochastic integrals are $L^2$ martingales with expectation $0$. Hence 
$$\lVert \alpha^m\rVert_{L^2(M)}^2=\lVert \alpha^n\rVert_{L^2(M)}^2+\lVert \alpha^m-\alpha^n\rVert_{L^2(M)}^2$$ so the sequence
$(\lVert \alpha^n\rVert_{L^2(M)})_{n\geq 1}$ is increasing and bounded thus has finite limit 
$\sup_{n\in\mathbb{N}} \lVert\alpha^n\rVert_{L^2(M)}^2<\infty$.
It follows from
$$\lVert \alpha^m-\alpha^n\rVert_{L^2(M)}^2=\lVert\alpha^m\rVert_{L^2(M)}^2-\lVert\alpha^n\rVert_{L^2(M)}^2$$
that $(\alpha^n)_{n\geq1}$ is Cauchy in $L^2(M)$ and converges to some $\alpha\in L^2(M)$ 
which is predictable and hence in $\mathcal{S}^2(\mathcal{G},M)$.
\end{proof}

\subsection{Existence of an approximating sequence of discrete expansions \label{sub-SMGdec}}

In view of the preceding results it is natural to wonder about the converse direction, namely, 
the necessity of the existence and convergence of the compensators and information drifts of a semimartingale 
for a subsequence of filtrations $\mathcal{G}^n$ converging to a filtration $\mathcal{G}$ 
when the same property is assumed for $\mathcal{G}$.
Such results are established in terms of weak convergence 
and Skorohod's $J^1$ topology in \cite{CMS} and \cite{CMV}.
Here we first show a pointwise convergence in $L^1$ of the compensators when the compensator $A$ for $\mathcal{G}$ 
is obtained non constructively from Theorem \ref{thm-stabSM}.  

\bigskip

\begin{theorem}\label{thm-MCPlus}
Let $(\mathcal{G}^n)_{n\geq1}$ be a \textbf{non-decreasing} sequence of right-continuous filtrations 
such that $\forall t\in I,\mathcal{G}_t\subset\bigvee_{n\geq1}\mathcal{G}_t^n$. If $M$ is a 
$\mathcal{G}^n$ semimartingale with decomposition $M=:M^n+A^n$ for every $n\geq1$,
and $M$ is a $\mathcal{G}$-semimartingale with decomposition $M=:\widetilde{M}+A$,
then 
$$\forall t\in I,A_t^n\xrightarrow[n\to\infty]{L^1} A_t.$$
\end{theorem}
\begin{proof}
Since $M$ is a continuous $\mathcal{F}$-adapted $\mathcal{G}$-semimartingale, the filtration shrinkage theorem 
in~\cite{FP} assures that $M$ is a $\mathcal{G}^n$-special semimartingale with decomposition 
$M={}^{(n)}\widetilde{M}+{}^{(n)}A$, where ${}^{(n)}$ stands for the optional projection with respect to $\mathcal{G}^n$. 

It follows from the uniqueness of the semimartingale decomposition of $M$ in the right-continuous filtration $\mathcal{G}^n$ that $A^n= {}^{(n)}A$, and Proposition \ref{prop-propConv} proves the pointwise convergence $\forall t\in I,\mathbb{E}[A_t|\mathcal{G}_t^n]\xrightarrow[n\to\infty]{L^1}A_t$.
\end{proof}

We would like to reach the same conclusion in term of convergence in $\mathscr{H}^1$ and $\mathscr{H}^2$ for $A$ 
and the information drift (instead of the integrated finite variation process). 
Based on the next Theorem \ref{thm-embedding},
Theorem \ref{thm-brownian-miracle}
shows that is possible in the case of square integrable martingales.

\bigskip

\begin{theorem} \label{thm-embedding}
Let $M$ be a $\mathcal{F}$-local martingale and $\mathcal{G}$ and $\mathcal{H}$ two filtrations with $\mathcal{F}_t\subset\mathcal{G}_t\subset\mathcal{H}_t$ for every $t\geq0$.
Suppose that $M$ is a $\mathcal{H}$-semimartingale with a decomposition $M=M_0+\widetilde{M}+\int_0^.\alpha_u d[M,M]_u$, with $\mathbb{E}\int_I|\alpha_u|d[M,M]_u <\infty$. 
Then $M$ is a $\mathcal{G}$ semimartingale with decomposition $M = M_0 + {}^o\widetilde{M} + \int {}^o\alpha_u d[M,M]_u$, where the optional projections are taken with respect to $\mathcal{G}$. In particular the optional projection of $\widetilde{M}$ on $\mathcal{F}$ is $M$.
\end{theorem}
\begin{proof}
The existence of optional projections is guaranteed as follows:
on the one hand, $M$ is $\mathcal{G}$-adapted and continuous hence $\mathcal{G}$-optional; 
on the other hand $A$ admits an optional projection (on $\mathcal{G}$) because our assumptions imply 
$\mathbb{E}\int_I|\alpha_u|d[M,M]_u <\infty$ so that is $A$ is (prelocally) integrable;
thus the equation $M = M_0 + \widetilde{M}+ \int_0^.\alpha_ud[M,M]_u$ gives the existence of the optional 
projection ${}^o\widetilde{M}$ as well as
\bee
M = M_0 + {}^o\widetilde{M}+ {}^o\int_0^.\alpha_ud[M,M]_u
\eee

The key argument of this proof is that we have $[\widetilde{M},\widetilde{M}] = [M,M]$. 
The theorem can be proved with classical arguments if $M$ is a square integrable martingale. 
For illustration purposes let us consider first the case where $[M,M]$ is bounded, hence integrable.
It follows from \cite[pp.73-74, Corollary 3]{PPbook} 
that the continuous $\mathcal{H}$-local martingale $\widetilde{M}$ is a square integrable martingale 
and in particular a continuous $\mathcal{H}$-martingale, so that its optional projection on 
$\mathcal{G}$ is a $\mathcal{G}$-martingale by Theorem 2.2 of \cite{FP} on filtration shrinkage.

In the general case, if $\mathbb{P}([M,M]_t\xrightarrow[n\to\infty]{}\infty)>0$, we can define a sequence of $\mathcal{F}$-stopping times $\tau_k:=\inf\{t\geq0: [M,M]_t > k\}$ which we can suppose increasing by considering $\widetilde{\tau}_0 := \tau_0$, $\widetilde{\tau}_{k+1}:=\tau_n\vee(\widetilde\tau_{k-1}+1)$.
Again, it follows from $[M,M] = [\widetilde{M},\widetilde{M}]$ that $[\widetilde{M},\widetilde{M}]^{\tau_k}$ is bounded. Hence \cite[pp.73-74, Corollary 3]{PPbook} still applies so that $[\widetilde{M},\widetilde{M}]^{\tau_k}$ is a square integrable $\mathcal{G}$-martingale. 
Hence $\tau_k$ is a $\mathcal{F}$-localizing sequence which turns $\widetilde{M}$ into a $\mathcal{G}$-martingale, which implies that the optional projection of $\widetilde{M}$ on $\mathcal{F}$ is a $\mathcal{F}$-local martingale by \cite[Theorem 3.7]{FP}. 
From there we obtain easily that the optional projection of $\widetilde{M}$ on $\mathcal{G}^n$ is a $\mathcal{G}^n$-local martingale, either by arguing that $\tau_k$ is also a $\mathcal{G}^n$-localizing sequence, or by invoking Corollary 3.2 of \cite{FP}.

Finally let us consider the last term $A_t:=\int_0^t\alpha_ud[M,M]_u$.
Since $$\mathbb{E}\int_I|\alpha_u|d[M,M]_u <\infty$$ 
it follows from the proof of Lemma 10 of \cite{KLP} that 
$$\mathbb{E}[\int_0^t\alpha_ud[M,M]_u|\mathcal{G}_t] - \int_0^t\mathbb{E}[\alpha_u|\mathcal{G}_u]d[M,M]_u$$
defines a $\mathcal{G}$-martingale (see also Theorem 57 and Remark 58.4.d of \cite{DelMey-B}), where $\mathbb{E}[\alpha_u|\mathcal{G}_u]$ is here any measurable adapted version.
A standard predictable projection argument finally lets us conclude the existence of a predictable version 
of $\alpha^n$. And, for every $n\geq 1$ we have by Jensen's inequality
$\mathbb{E}\int_0^T{}^o\alpha_s^2ds \leq \mathbb{E}\int_0^T\alpha_s^2ds<\infty$.
\end{proof}

\begin{theorem}\label{thm-brownian-miracle} Let $M$ be a continuous square integrable $\mathcal{F}$-martingale, $\mathcal{H}$ another filtration and $\mathcal{G}_t:=\bigcap_{u>t}(\mathcal{F}_u\vee\mathcal{H}_u)$. Suppose that $\mathcal{H}^n$ is a refining sequence of filtrations such that $\mathcal{H}_t = \bigvee_{n\geq1} \mathcal{H}_t^n$, and let $\mathcal{G}_t^n:=\bigcap_{u>t}(\mathcal{F}_u\vee\mathcal{H}_u^n)$. Then, the following statements are equivalent :
\begin{enumerate}[(i)]
\item There exists a predictable process $\alpha$ such that $\widetilde{M}_t:=M-\int_0^t\alpha_sds$ defines a continuous $\mathcal{G}$-local martingale and $\mathbb{E}\int_0^T\alpha_s^2ds<\infty$.
\item For every $n\geq1$ there exists a predictable process $\alpha^n$ such that $\widetilde{M}_t^n:=M-\int_0^t\alpha_u^n ds$ is a continuous $\mathcal{G}^n$-local martingale and 
$\sup_{n\geq1}\mathbb{E}\int_0^T(\alpha_u^n)^2d[M,M]_u$ 
$<\infty$.
\end{enumerate}
In that case, $\widetilde{M}$ (resp. $\widetilde{M}^n$) is a $\mathcal{G}$- (resp. $\mathcal{G}^n$-) square integrable martingale and we have $\int_0^t(\alpha_u^n-\alpha_u)^2 d[M,M]_u\xrightarrow[n\to\infty]{}0$.
\end{theorem}
\begin{proof}
Let us first remark that since $[\widetilde{M},\widetilde{M}] = [M,M]$
every $\mathcal{G}^n$-local martingale $\widetilde{M}^n$ or 
$\mathcal{G}$-local martingale $\widetilde{M}$ as in conditions $(i)$ and $(ii)$
is a square integrable martingale (see \cite[pp.73-74, Corollary 4]{PPbook}).
 
Clearly, since the condition $\mathcal{H}_t = \bigvee_{n\geq1} \mathcal{H}_t^n$ 
insures $\mathcal{G}^n\xrightarrow[n\to\infty]{L^2}\mathcal{G}$, the implication $(ii)\implies (i)$ 
as an application of Theorem \ref{thm-convID}.

Let us now suppose $(i)$, fix $n\geq 1$ and take a measurable adapted process $\alpha$ 
satisfying $\mathbb{E}\int_I|\alpha_u|d[M,M]_u <\infty$ such that 
$\widetilde{M}_t := M_t - \int_0^t\alpha_ud[M,M]_u$ is a local martingale.
In this equation, $M$ is $\mathcal{G}^n$-adapted and continuous hence $\mathcal{G}^n$-optional; 
since $\widetilde{M}$ is a continuous $\mathcal{G}$-martingale, its optional projection on $\mathcal{G}^n$
is a $\mathcal{G}^n$-martingale by Theorem 2.2 of \cite{FP} on filtration shrinkage.
Note that by Corollary 3.2 of \cite{FP}, proving that the optional projection of $\widetilde{M}$ on $\mathcal{F}$ is $M$
would be sufficient to obtain that $\widetilde{M}$ is a continuous $\mathcal{G}$-martingale.
Finally since $\mathbb{E}\int_I|\alpha_u|d[M,M]_u <\infty$ the last term $A_t:=\int_0^t\alpha_ud[M,M]_u$ 
it follows from the proof of Lemma 10 of \cite{KLP} that 
$\mathbb{E}[\int_0^t\alpha_ud[M,M]_u|\mathcal{G}_t^n] - \int_0^t\mathbb{E}[\alpha_u|\mathcal{G}_u]d[M,M]_u$
defines a $\mathcal{G}^n$-martingale (see also Theorem 57 and Remark 58.4.d of \cite{DelMey-B}), where $\mathbb{E}[\alpha_u|\mathcal{G}_u^n]$ is here any measurable adapted version.

A predictable projection argument finally lets us conclude the existence of a predictable version 
of $\alpha^n$. And, for every $n\geq 1$ we have by Jensen's inequality
$\mathbb{E}\int_0^T(\alpha_s^n)^2ds \leq \mathbb{E}\int_0^T\alpha_s^2ds<\infty$
which concludes concludes the proof of $(ii)$.

In both cases the convergence $\int_0^t(\alpha_s^n-\alpha_s)^2 ds\xrightarrow[n\to\infty]{}0$ 
follows from Theorem \ref{thm-convID} invoking the uniqueness of the limit in $\mathscr{H}^2$.
\end{proof}

\section{Expansion with a stochastic process}\label{sec-exp-process}

\noindent
Let us now consider an expansion of the filtration $\mathcal{F}$ with a c{\`a}dl{\` a}g process $X$ on a bounded time interval $I:=[0,T],T>0$. We let $\mathcal{H}$ be the natural filtration of $X$, 
$$\mathcal{H}_t=\sigma(X_s,s\leq t),$$
$\check{\mathcal{G}}$ the expansion of $\mathcal{F}$ with $X$, namely
$$\check{\mathcal{G}}_t:=\mathcal{F}_t\vee\mathcal{H}_t,$$
and we consider the smallest right-continuous filtration containing $\mathcal{F}$ to which $X$ is adapted $\mathcal{G}$, given by
$$\mathcal{G}_t:=\bigcap_{u>t}\check{\mathcal{G}}_u.$$
Proposition \ref{prop-usual-hyp} shows that $\mathcal{G}$ and $\check{\mathcal{G}}$ are essentially equivalent as regards semimartingales and decompositions.

As we argued in the introduction, expansions of filtrations with stochastic processes reaches a level of generality that one could desire in applications. Most quantitative applications would rely on an estimation of the compensator of semimartingales for the augmented filtration. Hence in a model of practical use it is natural to expect the existence of the information drift, for either non-degeneracy or tractability purposes.
Let us get some insight on a class of models that could satisfy this requirement.

Let us start by saying that adding independent information to the filtration does not have any effect on semimartingales and their decompositions (see \cite{BY}). Obviously, models of interest would involve some anticipation of the reference filtration. 
On the one hand it is now well known that one obtains non-degenerate models by initial expansions, even if they anticipate part of the reference filtration with a finite (and probably countable) number of random variables, as long as they satisfy Jacod's condition (see \cite{KLP} or \cite{PPbook}), so that we can expect to be able to pass to some continuous limit.

On the other hand, when one anticipates naively the future, it is easy to see how the semimartingale property can be broken: indeed if a local martingale remains a semimartingale while anticipating its future path on a fixed time interval,
then it must be of finite variation on compact sets (see also \cite{LNthesis}), which makes it for instance constant if it is continuous.

We deduce from the results of the preceding section that the expansion of $\mathcal{F}$ with $X$ can be consistently obtained as the limit of expansions with simple processes approximating $X$. A careful description of a class of processes satisfying our assumptions is given in Section \ref{sub-class-x}. We conclude with some illustrative examples.

\subsection{Discretization and convergence}

Kchia and Protter~\cite{KP} prove the conservation of the semimartingales in the above framework, and we add here a convergence result for the information drifts based on our convergence of filtrations in $L^p$ which is probably the most suggestive result of this article.

Let us approximate the expansion $\check{\mathcal{G}}$ with discrete augmentations generated by finite samples of $X$. Indeed such expansions have a behavior similar to initial expansions and can be relatively well understood under Jacod's condition which has been accepted to some extent as a satisfactory answer to the problem of initial expansion. For conciseness we will not give details on initial expansions (we refer the reader to \cite{KLP} and our introduction) and we state a corollary of Theorem 4 in \cite{KLP} as a preliminary result.\\ \\
Let $(\pi^n)_{n\geq 1}$ be a refining sequence of subdivisions of $[0,T]$ with mesh size converging to $0$. We also denote $\pi^n=:(t_i^n)_{i=0}^{\ell(n)}$, with $0=t_0^n<...<t_{\ell(n)}^n<t_{\ell(n)+1}^n=T$. For every $n\geq 1$ we define a discrete c{\`a}dl{\`a}g process $X^n$ by 
$$X_t^n:=X_t^{n,\pi}:=\sum_{i=0}^{\ell(n)} X_{t_{t_i^n}} 1_{t_i^n\leq t<t_{i+1}^n}.$$
We also define the non-decreasing sequence of filtration $\mathcal{H}^{n}$ generated by $X^n$, 
$$\mathcal{H}_t^n:=\mathcal{H}_t^{n,\pi}:=\sigma(X_{s}^n,s\leq t)=\sigma(X_{t_0^n},X_{t_1^n}-X_{t_0^n},...,X_{t_{\ell(n)+1}^n}-X_{t_{\ell(n)}^n}),$$
as well as the expansions $\check{\mathcal{G}}^n$ and $\mathcal{G}^n$ given by 
$$\check{\mathcal{G}}_t^n:=\mathcal{F}_t\vee\mathcal{H}_t^n\text{\hspace{.2in}and\hspace{.2in}}
\mathcal{G}_t^n=\bigcap_{u>t}\check{\mathcal{G}}_u^n.$$
Recall that $(E,\mathcal{E})$ is a generic Polish space and denote $(E,\mathcal{E})^n:=(E^n,\mathcal{E}^{\otimes n})$ 
the product Polish space ($n\geq 1$). 
Jacod introduces the next condition in \cite{Jacod-book} as a sufficient criterion for the conservation of semimartingales in initial expansions. (See also~\cite{PPbook}.)

\begin{proposition}[Jacod's condition] Let $L$ be a random variable 
and let
\begin{eqnarray*}
I\times \Omega \times \mathcal{A} &\longrightarrow &[0,1] \\
(t,\omega,A) & \longmapsto & P_t(\omega,A)
\end{eqnarray*}
be a regular version of its conditional distributions 
with respect to $\mathcal{F}$.
(see \cite[Theorem 58, p.52]{DelMey-B}).
$L$ is said to satisfy \textbf{Jacod's condition} if for almost every $t\in I$ there exists 
a (non-random) $\sigma$-finite measure $\eta_t$ on the Polish space $(E,\mathcal{E})$ in which it takes values 
such that 
$$P_t(\omega,L\in.)\ll\eta_t(.) \ a.s.$$
In this case $\eta$ can be taken to be constant with time, as well as the law of $L$ (\cite{Jacod-book}), 
and we can define the Jacod conditional density $q_t^L$ of $L$ as
\begin{equation}
q_t^L(\omega,x):=\frac{dP_t(\omega,.)}{d\eta(.)}\bigg|_{\sigma(L)} (L)^{-1}(x) \label{eq:Jac-density}
\end{equation}
\end{proposition}

This condition is extended to successive expansions at multiple times in \cite{KLP} 
of which the following proposition is a corollary in the particular case 
where the expansion times are deterministic.

\begin{proposition}\label{prop-KLP-to-process} 
Let $n\geq1$. Suppose that $\big(X_{t_0^n},X_{t_1^n}-X_{t_0^n},...,X_{t_{\ell(n)+1}^n}-X_{t_{\ell(n)}^n}\big)$ satisfies Jacod's condition and for $k\leq \ell(n)$ let $q^{k,n}$ be the conditional density of $L^k:=\big(X_{t_0^n},X_{t_1^n}-X_{t_0^n},...,X_{t_{k+1}^n}-X_{t_{k}^n}\big)$ defined by Equation \eqref{eq:Jac-density}. 
Then every continuous $\mathcal{F}$-local martingale $M$ is a ${\mathcal{G}}^n$-semimartingale on $[0,T]$
with decomposition
$$M=:M^n+\int_0^.\alpha_s^n d[M,M]_s,$$
where $\alpha^n$ is the $\mathcal{G}^n$-predictable process defined by
\begin{equation}
\alpha_t^n(\omega) :=\! \sum_{k=0}^{\ell(n)} 1_{t_k^n\leq t<t_{k+1}^n} \! 
\frac{1}{q_{s_-}^{k,n}(.,x_k)}
\frac{d[q^{k,n}(.,x_k),M]_s}{d[M,M]_s}\bigg|_{x_k=(X_{t_0},X_{t_1}\!-\!X_{t_0},\cdots,X_{t_k}\!-\!X_{t_{k-1}})} \label{eq:alpha-density}.
\end{equation}
and where the existence of the Radon-Nykodym derivative is ensured by the Kunita-Watanabe inequality.
\end{proposition}
\noindent
We emphasize that, taking $\eta$ to be the law of $L$, the conditional densities  $q^{k,n}$ can be computed explicitly from the conditional finite dimensional distributions of $X$ as
\begin{align*}
q_t^{k,n}(\omega,x):=\frac{P_t\big(\omega,(X_{t_0^n},X_{t_1^n}-X_{t_0^n},...,X_{t_{k+1}^n}-X_{t_{k}^n})\in dx\big)}{\mathbb{P}\big((X_{t_0^n},X_{t_1^n}-X_{t_0^n},...,X_{t_{k+1}^n}-X_{t_{k}^n})\in dx\big)}
\end{align*}

We apply here the results of Section \ref{sec-cv-fil} on convergence of filtrations together with Proposition \ref{prop-KLP-to-process} to obtain the semimartingale property and decomposition of $M$ in the filtration $\mathcal{G}$ augmented with the process $X$. \\ \\
We postpone to Paragraph \ref{sub-class-x} the discussion on the minor technical assumptions required on the process $X$ to fit in our framework by introducing the class $\mathbb{X}^\pi$, where $\pi$ is a discretization scheme. 

\begin{definition}[Class $\mathbb{X}^\pi$]\label{def-class-x} We say that the c{\`a}dl{\`a}g process $X$ is of Class $\mathbb{X}^{\pi}$, or of Class $\mathbb{X}$ if there is no ambiguity, if
$$\forall t\in I,\mathcal{H}_{t^-}^n\xrightarrow[n\to 1]{L^2}\mathcal{H}_{t^-}.$$
\end{definition}

\begin{proposition} The c{\`a}dl{\`a}g process $X$ is of Class $\mathbb{X}$ if one of the following holds:
\begin{enumerate}[(i)]
\item $\mathbb{P}(\Delta X_t\neq0)=0$ for any fixed time $t>0$.
\item[(i')] $X$ is continuous
\item $\mathcal{H}$ is (quasi-) left continuous.
\item[(ii')] $X$ is a Hunt process (e.g a L\'evy process)
\item $X$ jumps only at totally inaccessible stopping times
\item $\pi$ contains all the fixed times of discontinuity of $X$ after a given rank
\end{enumerate}
\end{proposition}
\begin{proof}
See Paragraph \ref{sub-class-x}.
\end{proof}

\noindent
Let us emphasize that that all conditions except the last depends only on the properties of $X$, and not on the choice of the discretization $\pi$. The next lemma extends the convergence of the filtrations $\mathcal{H}^n$ to the filtrations $\mathcal{G}^n$.

\begin{lemma}\label{lemma-conv-Gn}
If $X$ is of Class $\mathbb{X}$, then $\mathcal{G}^n\xrightarrow[n\to\infty]{L^2}\check{\mathcal{G}}$.
\end{lemma}
\begin{proof}
It follows from the definition of Class $\mathbb{X}$ that $\mathcal{H}^n\xrightarrow[n\to\infty]{L^2}{\mathcal{H}}$. By Proposition \ref{prop-YnvZn-to-YvZ} we also have $\check{\mathcal{G}}^n\xrightarrow[n\to\infty]{} \check{\mathcal{G}}$, and by Corollary \ref{cor-doncDiscRC} the convergence extends to the right-continuous filtrations $\mathcal{G}^n$ and $\mathcal{G}^n\xrightarrow[n\to\infty]{L^2}\check{\mathcal{G}}$.
\end{proof}

We can now state our main theorem.

\begin{theorem}[Expansion with a stochastic process] \label{thm-Tprocess}
Let $X$ be a stochastic process of Class $\mathbb{X}^\pi$ for some sequence of subdivisions 
$\pi^n:=(t_i^n)_{i=1}^{\ell(n)}$ of $[0,T]$ and $M$ a continuous $\mathcal{F}$-local martingale. Suppose that for every $n\geq1$ the random variable 
$\big( X_{0}, X_{t_1^n} - X_{t_0^n}, ..., X_{T}-X_{t_{\ell(n)}^n}\big)$ satisfies Jacod's condition 
and let $\alpha^n$ be a $\mathcal{G}^n$-predictable version of the process
$$\alpha_t^n(\omega) :=\! \sum_{k=0}^{\ell(n)} 1_{t_k^n\leq t<t_{k+1}^n} \! 
\frac{1}{q_{s_-}^{k,n}(.,x_k)}
\frac{d[q^{k,n}(.,x_k),M]_s}{d[M,M]_s}\bigg|_{x_k=(X_{t_0},X_{t_1}\!-\!X_{t_0},\cdots,X_{t_k}\!-\!X_{t_{k-1}})}$$
as defined by Equation \eqref{eq:alpha-density} in Proposition \ref{prop-KLP-to-process}. Then:
\begin{enumerate}[1.] 
\item \label{pt-T2-1} If $sup_{n\geq1}\mathbb{E}\int_0^T|\alpha_t^n|d[M,M]_t<\infty$, $M$ is a continuous $\mathcal{G}$-semimartingale.
\item \label{pt-T2-2} If $\sup_n\mathbb{E}\int_0^T(\alpha_s^n)^2d[M,M]_s<\infty$, $M$ is a continuous $\mathcal{G}$-semimartingale with decomposition 
$$M=:\widetilde{M}+\int_0^.\alpha_sd[M,M]_s$$
where $\alpha\in\mathcal{S}^2(\mathcal{G},M)$ and $\mathbb{E}\int_0^T (\alpha_t^n-\alpha_t)^2d[M,M]_t\xrightarrow[n\to\infty]{}0$.
\end{enumerate}
\end{theorem}

\noindent
We recall that $\alpha\in\mathscr{H}^2(\mathcal{G},M)$ means that $\alpha$ is $\mathcal{G}$-predictable and $\mathbb{E}\int_0^T \alpha_t^2d[M,M]_s<\infty$.

\begin{proof}[Proof of the theorem]
Proposition \ref{prop-KLP-to-process} shows that, for every $n\geq 1$, $M$ is a $\mathcal{G}^n$ semimartingale with decomposition $M=:M^n+\int_0^.\alpha_s^nd[M,M]_s$ for a $\mathcal{G}^n$-predictable process $\alpha^n$. Lemma \ref{lemma-conv-Gn} establishes the convergence $\mathcal{G}^n\xrightarrow[n\to\infty]{L^2}\mathcal{G}$, and the two points of the theorem now follow respectively from Theorem \ref{thm-stabSM} and Theorem \ref{thm-convID}.
\end{proof}

Let us now consider a new dynamical example inspired by the 1973 movie 
``The Sting'' by George Roy Hill, and starring Paul Newman and Robert Redford. We can think of it as a metaphor for high-frequency trading.
In the movie, two con men rip off horse gamblers 
by ``past posting'' bets after the closure.
In the context of high-frequency trading it is easy to imagine situations 
where the informational advantage could amount to a fixed advance in the future. This is appropriate for examples such as colocation, data frequency or a number of intermediate agents.
With this perspective we will suppose here 
that a given agent can see a small time $\delta$ in the future.
Of course, there is no hope that a local martingale would
remain a semimartingale if its future can be fully observed;
therefore in our models the anticipative signal needs to be corrupted with some noise. This seems realistic in any event.

\begin{example}
\label{ex-path-anticipation}
Let $X_t:=W_{t+\delta}+\epsilon_t$ 
where
$W$ is a Brownian motion and
 $\epsilon$ a stochastic process 
with marginals absolutely continuous with respect to Lebesgue measure
on a time interval.
For a given $t\geq 0$ we let also $(t_i^n)_i^n$ be 
a refining sequence of subdivisions of $[0,t]$
and denote $\phi_k$ the joint density of 
$\epsilon_{t_1},...,\epsilon_{t_k}$
as well as $\widetilde{\psi}_{k,i,n}(u_{k+1}, ..., u_i | u_{1}, ..., u_k)$ 
the conditional density of
$\epsilon_{t_{k+1}},...,\epsilon_{t_i}$
given $\epsilon_{t_1},...,\epsilon_{t_k}$
and
${\psi}_{k,n}(u_{k+1}| u_{1}, ..., u_k)$ 
the conditional density of
$\epsilon_{t_{k+1}}$
given $\epsilon_{t_1},...,\epsilon_{t_k}$.
We have
\begin{eqnarray*}
p_{s}^{i,n}(x) & := & P(X_{t_1}\leq x_1, ..., X_{t_i} \leq x_i | \mathcal{F}_s) \\
& = &
P(W_{t_1} + \epsilon_{t_1}\leq x_1, ..., W_{t_i} + \epsilon_{t_i}\leq x_i | \mathcal{F}_s) \\
& = & 
\iint...\int 
P(
W_{t_1} + u_1\leq x_1, ..., W_{t_i} + u_i\leq x_i 
| \mathcal{F}_s) \phi_i(u_1, ..., u_i) du_1...du_i
\end{eqnarray*}
Now suppose that $t_k\leq s < t_{k + 1}$. 
It is clear that for $ i\leq k$ the conditional density is
$p_s^{i,n}(x) = \phi_i(x_1-W_{t_1},...,x_i - W_{t_i})$;
for $i > k$,

\begin{eqnarray*}
&& P(X_{t_1}\leq x_1, ..., X_{t_i} \leq x_i | \mathcal{F}_s) \\
& = &
\iint...\int 
P(
W_{t_1} + u_1\leq x_1, ..., W_{t_k} + u_k\leq x_k 
| \mathcal{F}_s) \times\\
&&P(
W_{t_{k+1}} + u_{k+1}\leq x_{k+1}, ..., W_{t_n} + u_i\leq x_i 
| \mathcal{F}_s) \times\\
&& 
\phi_k(u_1, ..., u_k) 
\widetilde{\psi}_{k,i,n}(u_{k+1}, ..., u_i | u_{1}, ..., u_k)
du_1...du_k du_{k+1}...du_i \\
& = &
\iint...\int 
\Bigg[
\iint...\int 
 P(W_{t_{k+1}} + u_{k+1}\leq x_{k+1}, ..., W_{t_i} + u_i\leq x_i 
| \mathcal{F}_s) \times\\
&&\widetilde{\psi}_{k,i,n}(u_{k+1}, ..., u_i | u_{1}, ..., u_k)
du_{k+1}...du_i
\Bigg]
1_{W_{t_1} + u_1\leq x_1} ... 1_{W_{t_k} + u_k\leq x_k} 
\phi_k(u_1, ..., u_k) 
du_1...du_k. 
\end{eqnarray*}
It follows that for $i>k$ the conditional density is
{
\begin{align*}
p_s^{i,n}(x) 
 & =  
\phi_{k}(x_1-W_{t_1},...,x_k - W_{t_k})
\int\int...\int 
 \pi(t_{k+1} - s, W_s, x_{k+1} - u_{k+1}) 
 ... \\
 & ... \ \pi(t_i - t_{i-1}, x_{i-1} + u_{i-1}, x_i - u_i) 
 \widetilde{\psi}_{k,i,n}(u_{k+1}, ..., u_n | x_1-W_{t_1}, ..., x_k-W_{t_k})
 du_{k+1}...du_i
\end{align*}
}%
In fact we are interested in the limit
$\lim_{n\to\infty}\sum_{k=1}^n (p_{t_{k+1}} - p_{t_k}) (W_{t_{k+1}} - W_{t_k})$ so that we are interested in the interval $t_{i}\leq s< t_{i+1}$
on which we have the expression
\begin{equation*}
p_{s}^{i,n}(x) 
=
\phi_{i}(x_1-W_{t_1},...,x_i - W_{t_i})
\int 
 \pi(t_{i+1} - s, W_s, x_{i+1} - u_{i+1}) 
 \psi_{i,n}(u_{i+1} | x_1-W_{t_1}, ..., x_i-W_{t_i})
 du_{i+1}.
 \end{equation*}
 Therefore the information drift generated by discrete expansion with $X_{t_i}$ at each time $t_i$ exists and is given by
 \begin{eqnarray*}
 \alpha_t^n
 & := &
 \frac{d[p^{i,n},W]_t}{p_t^{i,n}(X) dt} \\
& = &
\int 
 \frac{1}{\pi}\frac{\partial\pi}{\partial x} (t_{i+1} - s, W_s, X_{t_{i+1}} + u_{i+1}) 
 \psi_{i,n}(u_{i+1} | \epsilon_{t_1}, ..., \epsilon_{t_i})
  du_{i+1} 
 \end{eqnarray*}
Thus the potential candidate for the $\mathcal{G}$-information drift is
\begin{equation*}
 \lim_{s\to t}
  \int 
  \frac{X_t + u - W_s}{t-s}
 \psi_{t}(u | \epsilon_r,r\leq s)
  du
\end{equation*} 
Precisely Theorem \ref{thm-Tprocess} shows that the limit information drift exists if the above limit is well defined in $L^2$. 
For instance, if $\epsilon$  is a Markov process with stationary increments with distribution $\kappa_{t-s}$ then
$\psi_t(u|\epsilon_r,r\leq s)
=
\kappa_{t-s}(u - \epsilon_s)$
and the above becomes
\begin{equation*}
 \lim_{s\to t}
  \int 
  \frac{X_t + u - W_s}{t-s}
 \kappa_{t-s}(u - \epsilon_s)
  du.
\end{equation*} 
The last expression highlights that 
the decay speed of the 
noise plays a key role in 
whether the information drift 
exists: sufficient noise alone
might be able to guarantee the existence of the information drift.
This was already suggested in previous analyses of
initial expansions with noise (e.g. \cite{Corcuera}).
\end{example}

\begin{remark}
Example \ref{ex-path-anticipation} shows that the correlation structure 
of the noise term $\epsilon_t$ has implications 
on the existence, or nonexistence, of an information drift.
This is also illustrated by the extreme case of white noise:
if $\epsilon_t$ in Example \ref{ex-path-anticipation} 
are i.i.d. centered random variables (generalized white noise), 
then for any non-decreasing sequence $(h_n)_{n\geq 1}$ with
$h_n\to\infty$ and $\frac{h_n}{n}\to 0$
the strong law of large numbers leads to
$$\frac{1}{n} \sum_{i=1}^n 
(W_{t - \frac{h_i}{n} + \delta} + \epsilon_{t-\frac{h_i}{n}})
\xrightarrow[n\to\infty]{a.s.}
W_{t+\delta}.$$
Since the terms of the sequence are all $\mathcal{G}_t$-measurable,
this implies that $W_t+\delta$ must be $\mathcal{G}_t$-measurable 
as well.
Thus, following the remark preceding Example \ref{ex-path-anticipation},
no information drift is possible in the case of additive white noise.
[Note: this argument was suggested by Jean Jacod.]
\end{remark}

\noindent Theorem \ref{thm-MCPlus} leads to a partial result in the converse direction.

\begin{proposition}
If $M$ is a $\mathcal{G}$-semimartingale with decomposition \\
$M=:\widetilde{M}+A$, then we have the pointwise convergence:
$$\forall t\geq 0,A_t^n\xrightarrow[n\to\infty]{L^1}A_t,$$ where $A_t^n:=\int_0^t\alpha_s^nd[M,M]_s,t\in I$.
\end{proposition}

Theorem \ref{thm-Tprocess} can actually be applied to obtain
sufficient conditions on the process $X$ ensuring the existence 
of an information drift for the expansion
$\mathcal{G}$ of $\mathcal{F}$ with $X$,
independently of any discretization.
Given a stochastic process $X$ and $t\geq 0$ we will denote
$X^t$ the stopped process given by $X_s^t = X_{s\wedge t}$.
\begin{theorem} 
Let $X$ be a stochastic process satisfying a Jacod type condition 
on the Wiener space
$\mathbb{P}(X^t\in dx|\mathcal{F}_t.) \ll
\mathbb{P}(X^t\in dx)$
and let 
$q_t(., x) = \frac{\mathbb{P}(X^t\in dx|\mathcal{F}_t.)}{\mathbb{P}(X^t\in dx)}$.
$q_t(.,x)$ is a martingale for every $x$ and we can define 
by the Kunita-Watanabe inequality
\bee
\alpha_t = \frac{d[\log q_t(., x), M]_t}{d[M, M]_t}\bigg|_{x = X^t} \ a.s.
\eee
If $\mathbb{E}\int\alpha_s^2d[M,M]_s<\infty$
then $\alpha$ is the information drift of $\mathcal{G}$.
\end{theorem}
\begin{proof}
That $q(.,x)$ defines a martingale for every $x$ has been well known
\cite[Lemma 1.8]{Jacod-book}, and the Kunita-Watanabe inequality can be found in \cite{PPbook}.
Let $\alpha_t = \frac{d[\log q_t(., x), M]_t}{d[M, M]_t}\bigg|_{x = X^t}$.
Let also $X^n$, $\mathcal{G}^n$ be as in the beginning of Section 4.1.
It is clear that $X^n$ satisfies Jacod condition,
and thus 
by Proposition \ref{prop-KLP-to-process} admits
an information drift $\alpha^n$ with respect to $\mathcal{G}^n$,
as defined in Equation \eqref{eq:alpha-density}.
The conclusion follows from Theorem \ref{thm-Tprocess}
as long as we have 
$\mathbb{E}\int_0^t(\alpha_s^n - \alpha_s)^2 d[M, M]_s
\xrightarrow[n\to\infty]{} 0$.
Let us denote $X_k:=(X_{t_0},X_{t_1}\!-\!X_{t_0},\cdots,X_{t_k}\!-\!X_{t_{k-1}})$. Since we have
\begin{align*}
&\mathbb{E}\int_0^t(\alpha_s^n - \alpha_s)^2 ds \\
=  
&\sum_{k = 1}^{\ell(n)} \mathbb{E}\int_{t_k^n}^{t_{k+1}^n} 
\Bigg(
\frac{d[\log q_t^{k, n}(., x_k), M]_t}{d[M, M]_t}
\bigg|_{x_k=X_k} 
-  \frac{d[\log q_t(., x), M]_t}{d[M, M]_t}^2\bigg|_{x=X^t} 
\Bigg)^2 ds,
\end{align*}
 we only need to show that 
\bee
\frac{d[\log q_t^{k, n}(., x_k), M]_t}{d[M, M]_t}
\bigg|_{x_k=X_k} 
1_{t_k\leq t < t_{k+1}}
\xrightarrow[n\to\infty]{L^2(M)}  
\frac{d[\log q_t(., x), M]_t}{d[M, M]_t}\bigg|_{x = X^t}
\label{eq-conv}
\tag{$*$}
\eee
uniformly with respect to $k$.\\ \\
For every measurable set $A$, the convergence 
$\int_A q_t^{k, n}(., x) dx \xrightarrow[n\to\infty]{a.s.,L^2}
\int_A q_t(x) dx$
is ensured by the martingale convergence theorem for
bounded martingales, each $n$.
In particular we have the pointwise convergence 
$q_t^{k, n}\xrightarrow[n\to\infty]{L^2} q_t$.
This implies, by Theorem 3.3 of \cite{JacMelProt}
that the $t-$martingales representations 
of $q_t^{k, n}$ 
converges to the martingale representation of $q_t$
in $L^2(d\mathbb{P}\times d[M, M]_t)$.
\end{proof}

\subsection{Description of class $\mathbb{X}$} \label{sub-class-x}

In this paragraph we derive several sufficient conditions for $X$ to be of class $\mathbb{X}$. A particular subset of processes of class $\mathbb{X}$ is easy to characterize.

\subsubsection{Class $\mathbb{X}_0$}

\begin{definition}
We say that $X$ is of Class $\mathbb{X}_0$ if $\mathcal{H}_{t^-}=\bigvee_{n\geq 1}\mathcal{H}_{t^-}^n$.
\end{definition}

\begin{proposition} If $X$ is of Class $\mathbb{X}_0$, it is also of Class $\mathbb{X}$
\end{proposition}
\begin{proof} This is a direct corollary of Proposition \ref{prop-VHn-providentiel-martingale}, which shows that if $\mathcal{H}_{t}=\bigvee_{n\geq 1}\mathcal{H}_{t}^n$, then for every $n\geq 1$ $\mathcal{H}_{t}^n\xrightarrow[n\geq 1]{L^p}\mathcal{H}_{t}$. Note that we need not assume that $\pi^n$ contains all fixed times of discontinuity of $X$ as is required in Lemma 5 and Example 1 of \cite{KP}.
\end{proof}

\begin{proposition} The process $X$ is of Class $\mathbb{X}_0$ if its natural filtration $\mathcal{H}$ is quasi left-continuous.
\end{proposition}
\begin{proof}
In general we have the inclusion
$$\mathcal{H}_{t^-}\subset\bigvee_{n\geq1}\mathcal{H}_t^n\subset\mathcal{H}_{t}.$$ It follows that $X$ is of Class $\mathbb{X}_0$ if $\mathcal{H}$ is left continuous. More generally the second inclusion is tight if $\mathcal{H}$ is quasi-left continuous. This is for instance the case if $X$ is a L{\'e}vy, Feller or more generally a c{\`a}dl{\`a}g Hunt Markov process, or any integral of such a process.
\end{proof}

\begin{corollary} Hunt processes (which include L{\'e}vy processes) are of class $\mathbb{X}_0$.
\end{corollary}
\begin{proof}
See for instance \cite[p. 36]{PPbook}. Note that if $X$ is a Hunt process, the $\mathbb{P}$-completion of $\mathcal{H}$ is also right-continuous, but we cannot in general conclude that $\check{\mathcal{G}}$ is also right-continuous. 
\end{proof}

\subsubsection{Class $\mathbb{X}$}

This paragraph shows that class $\mathbb{X}$ contains c{\`a}dl{\`a}g processes without fixed points of discontinuity (in particular continuous processes) and Hunt processes. We start with a characterization of the convergence of filtrations in $L^p$ with finite dimensional distributions for filtrations generated by c{\`a}dl{\`a}g process.

\begin{proposition}\label{prop-cv-fun-to-filtr}
Let $\mby:=\sigma(Y_t,0\leq t\leq T)$ for some c{\`a}dl{\`a}g measurable process $Y$ and $p\geq 1$. \\ 
For a given sequence $\mby^n,n\geq1$ of $\sigma$-algebras, $\mby^n\xrightarrow[n\to\infty]{L^p}\mby$ if and only if for all $k\in\mathbb{N}$,$u_1,...,u_k\in[0,T]$ and $f$ bounded continuous function we have
$\mathbb{E}[f(Y_{u_1},...,Y_{u_k})|\mby^n]\xrightarrow[n\to\infty]{L^p}f(Y_{u_1},...,Y _{u_k})$.
\end{proposition}
\begin{proof} The proof is identical to the characterization of weak convergence in Lemma 3 of \cite{KP}.
\end{proof}

\begin{proposition}\label{prop-cv-fdd-to-filtr} Let $Y,Y^n,n\geq 1$ be c{\`a}dl{\`a}g stochastic processes and $\mby,\mby^n$ their respective natural filtrations.
Suppose that
\begin{equation*}
\forall k\in\mathbb{N},0\leq u_1\leq ... \leq u_k, (Y_{u_1}^n,Y_{u_2}^n,...,Y_{u_k}^n)\xrightarrow[n\to\infty]{d}(Y_{u_1},Y_{u_2},...,Y_{u_k})
\end{equation*}
Then for every $p\geq 1$ we have $\mby^n\xrightarrow[n\to\infty]{L^p}{\mby}$.
\end{proposition}
\begin{proof}
Fix $p\geq 1,0\leq u_1\leq ... \leq u_k\leq t$ and let $f$ a bounded continuous function. By the Minkowski inequality (and because $f(Y_{u_1}^n,Y_{u_2}^n,...,Y_{u_k}^n)$ is $\mby_t^n$-measurable) we have
\begin{align*}
&\Big| \mathbb{E}[f(Y_{u_1},Y_{u_2},...,X_{u_k})|\mby_t^n]
- f(Y_{u_1},Y_{u_2},...,Y_{u_k}) \Big|^p \\
&\leq 
\Big| \mathbb{E}[f(Y_{u_1},Y_{u_2},...,Y_{u_k})
- f(Y_{u_1}^n,Y_{u_2}^n,...,Y_{u_k}^n)|\mby_t^n] \Big|^p\\
&+ \Big| f(Y_{u_1}^n,Y_{u_2}^n,...,Y_{u_k}^n)
- f(Y_{u_1},Y_{u_2},...,Y_{u_k}) \Big|^p 
\end{align*}
and consequently
\begin{align*}
&\mathbb{E}\Big| \mathbb{E}[f(Y_{u_1},Y_{u_2},...,Y_{u_k})|\mby_t^n]
- \mathbb{E}[f(Y_{u_1},Y_{u_2},...,Y_{u_k})|{\mby}_t]\Big| \\
&\leq 2\mathbb{E}[\Big| f(Y_{u_1},Y_{u_2},...,Y_{u_k})
- f(Y_{u_1}^n,Y_{u_2}^n,...,Y_{u_k}^n)\Big|].
\end{align*}
Since $f$ is bounded, for every $p\geq 1$, $f(Y_{u_1},Y_{u_2},...,Y_{u_k}) \xrightarrow[n\to\infty]{L^p} f(Y_{u_1}^n,Y_{u_2}^n,...,Y_{u_k}^n)$ as long as $(Y_{u_1}^n,Y_{u_2}^n,...,Y_{u_k}^n)\xrightarrow[n\to\infty]{d}(Y_{u_1},Y_{u_2},...,Y_{u_k})$
so we can conclude from Proposition \ref{prop-cv-fun-to-filtr}.
\end{proof}

\noindent 
Using this characterization, we show that convergence in probability of the underlying processes is sufficient for convergence of filtrations, under a technical assumption on the jumps.

\bigskip

\begin{proposition} \label{prop-cont-to-class-d} If $X$ is a c{\`a}dl{\`a}g process with $\Delta X_t=0$ a.s. for every $t\in I$, then $X$ is of class $\mathbb{X}$.
\end{proposition}
\begin{proof}  The hypothesis of Proposition \ref{prop-cv-fdd-to-filtr} is satisfied as long as we have $\forall t\in I$, $Y_t^n\xrightarrow[n\to\infty]{\mathbb{P}}Y_t$ and $\mathbb{P}(\Delta Y_t\neq 0)$. With our discretization setting this is a consequence of the convergence $X^n\xrightarrow[n\to\infty]{a.s.}X$ (which actually holds if $\pi^n$ contain the fixed points of discontinuity of $X$ after some fixed rank).
\end{proof}

\begin{remark} The proof of the last proposition can be put in perspective with Lemmas 3 and 4 of \cite{KP} which derive sufficient conditions for weak convergence of the filtrations generated by c{\`a}dl{\` a}g processes, where the underlying processes are required to converge in probability for the Skorohod $J^1$ topology. In comparison we only ask for pointwise convergence.
\end{remark}

\subsection{Examples}

\paragraph{It{\^o}'s initial example \cite{Ito}.}
When $\mathcal{F}$ is the filtration generated by a Brownian motion $W$ and $\mathcal{G}$ its initial expansion with $W_1$, it is classical that there exists an information drift on $[0,1)$ given by $\frac{W_1 - W_s}{1-s}$,
 which can be extended to time $1$. 
 However $\mathbb{E}[\alpha_s^2] = \frac{\mathbb{E}[(W_1 - W_s)^2}{(1-s)^2} = \frac{1}{1-s}$ 
 with $\int_0^t\mathbb{E}[\alpha_s^2]  ds = 1 - log(1-t) \xrightarrow[t\to 1]{} \infty$
This implies that although the information drift exists in $\mathscr{H}^1$ we cannot define 
$\alpha\in\mathscr{H}^2$ on $[0,1]$.

\paragraph{A suggestive example.} 
Let $\mathcal{F}$ continue to be the filtration 
generated by a Brownian motion $W$ and
let us take here $X_t := W_1 + \epsilon\widetilde{W}_{1-s}^H$ where $\widetilde{W}^H$ is a fractional 
Brownian motion with Hurst parameter $H$ independent from $W$. Variations of this setup have 
been studied in several articles including \cite{KP}, \cite{Corcuera}, \cite{Ank-thesis} and \cite{LNthesis}.
It is easy to show using Gaussian calculations (e.g. \cite{KP}) or more general representations (\cite{Corcuera})
that $W$ admits an information drift for the filtration $\mathcal{G}_t:=\mathcal{F}_t\vee\sigma(X_u, u\leq t)$ 
which is given by
$$\alpha_t = \frac{W_1-W_s + \epsilon \widetilde{W}_{1-s}^H}{1-s + \epsilon^2(1-s)^{2H}}.$$
This process converges in $\mathscr{H}^1$ for every $H$, but
$\mathbb{E}\alpha_s^2 = \frac{1}{1-s + \epsilon^2(1-s)^{2H}}$ so that 
$$\mathbb{E}\int_0^1\alpha_s^2ds < \infty \iff H<\frac{1}{2}$$
which corresponds to a fractional noise process with negatively correlated increments.

In the perspective of our search for a good model, we can observe that perturbing the anticipative signal with an independent noise doesn't change the nature of the semimartingale property, but changes the information drift and its integrability (see also \cite{LNthesis}).

\paragraph{The Bessel-3 process.}
Let us now consider the classical counterexample of the Bessel-3 process: let $Z$ be a Bessel-3 process, $\mathcal{F}$ 
its natural filtration, $X_t:=\inf_{t>s}Z_s$, and $\mathcal{G}$ the expansion of $\mathcal{F}$ with $X$. 
It is classical that $W_t:=Z_t-\int_0^t\frac{ds}{Z_s}$ is a $\mathcal{F}$- Brownian motion.
The formula of Pitman \cite{Pitman} shows that $W$ is a $\mathcal{G}$-semimartingale with decomposition
$W_t-(2X_t-\int_0^t\frac{ds}{Z_s})$ defines a $\mathcal{G}$-Brownian motion, which implies
that $W$ is a $\mathcal{G}$-semimartingale but cannot admit an information drift 
since the finite variation component is singular with respect to Lebesgue measure.
The non-constructive argument in Paragraph 6.2.1 in \cite{KP} shows independently that $W$ is a 
$\mathcal{G}$-semimartingale but does not allow to conclude regarding the existence of the information drift.
Our results allow this additional conclusion.

Similarly as in \cite{KP} let us use the discretization induced by the refining family of random times
 $\tau_p^n:=\sup\{t:Z_t=p\epsilon_n\}$. Due to the exclusivity of the indicators in the right sum below, we have
\begin{eqnarray*}
\alpha_s^n & = & \frac{1}{Z_s} - \sum_{p=0}^{\infty}1_{\tau_p^n<s}1_{Z_s\leq (p+1)\epsilon_n} \frac{1}{Z_s-p\epsilon_n}\\
\alpha_s^n & = & 
\frac{1}{Z_s^2} 
- \frac{2}{Z_s} \sum_{p=0}^{\infty}1_{\tau_p^n<s}1_{Z_s\leq (p+1)\epsilon_n} \frac{1}{Z_s-p\epsilon_n}
+ \sum_{p=0}^{\infty}1_{\tau_p^n<s}1_{Z_s\leq (p+1)\epsilon_n} \frac{1}{(Z_s-p\epsilon_n)^2}\\
\end{eqnarray*}
 and it follows from the identity 
 $\mathbb{P}(\tau_p^n < s|\mathcal{F}_s) 
 = (1-\frac{p\epsilon_n}{Z_s})_+$
 that
\begin{eqnarray*}
\mathbb{E}[(\alpha_s^n)^2] 
& = & \mathbb{E}[\frac{-1}{Z_s^2} +\sum_{p=0}^\infty 1_{p\epsilon_n\leq Z_s\leq (p+1)\epsilon_n}\frac{1}{Z_s(Z_s-p\epsilon_n)}] \\
& = &
\mathbb{E}[\sum_{p=0}^\infty 1_{p\epsilon_n\leq Z_s\leq (p+1)\epsilon_n}\frac{p\epsilon_n}{Z_s^2(Z_s-p\epsilon_n)}]\\
& \geq & \mathbb{E}[\sum_{p=0}^\infty 1_{p\epsilon_n\leq Z_s\leq (p+1)\epsilon_n} \frac{p}{Z_s^2}] \\
& \geq & \mathbb{E}[1_{\epsilon_n\leq Z_s} \frac{1}{Z_s^2} ]
\end{eqnarray*}
Hence $\int_0^t\mathbb{E}[1_{\epsilon_n\leq Z_s} \frac{1}{Z_s^2}] ds\xrightarrow[n\to\infty]{} \infty$and  the sequence of information drifts induced by the discretization satisfies $\sup_{n\geq 1} \mathbb{E}\int_0^\infty(\alpha_s^n)^2 ds = \infty$, so $\alpha^n$ cannot converge in $\mathscr{H}^2$ and there cannot exist an information drift $\alpha\in\mathscr{H}^2$.

\paragraph{Random anticipation.} 
In order to find expansions with a continuous anticipation satisfying the information drift property we must consider the right speed of anticipation.

\noindent
Let $\phi$ be a continuous time change, i.e. a non-decreasing stochastic process with continuous paths, independent from ${W}$ and let $X_t:=W_{t\wedge\phi_t}$.
The natural expansion of $\mathcal{F}$ with the process $X$ is the filtration $\check{\mathcal{G}}$ given by $\check{\mathcal{G}}_t:=\mathcal{F}_t\vee\sigma(X_s,s\leq t)$. It is equivalent (see Proposition \ref{prop-usual-hyp}) and useful for applications to consider the right-continuous filtration
$\mathcal{G}_t:=\bigcap_{u>t}(\mathcal{F}_u\vee\sigma(X_s,s\leq u)),t\in I$.

\begin{lemma} \label{lemma-tau}
Let $s\leq t$ and $\tau(s,t):=\inf\{0\leq u\leq s:u\vee\phi_u=t\}$ and define the pseudo-inverse $\phi^{-1}_u:=\inf\{v\geq 0 : \phi_v = u\}$. Then we have
\begin{enumerate}[(i)]
\item $\tau(s,t):=\phi^{-1}_t1_{t\leq\phi_s}+(s\vee\phi^{-1}_s)1_{t>\phi_s}$
\item $\mathbb{E}[W_t|\mathcal{G}_s]=\mathbb{E}[W_{\phi_{\tau(s,t)}}|\mathcal{G}_s]$
\item $\mathbb{E}[W_t|\mathcal{G}_s]=\int_0^s W_{\phi_u} \mathbb{P}(\tau(s,t)\in du|\mathcal{G}_s)$
\end{enumerate}
\end{lemma}
\begin{proof}
$(i)$ The formula for $\tau(s,t)$ follows from dissociation of cases. Note that it is continuous at $\phi_s=t$ as $s\leq t = \phi_s\implies s\vee\phi^{-1}_s=s=\phi^{-1}_t$.\\
$(ii)$ By the definition of $\tau(s,t)$, we have $W_{\phi_{\tau(s,t)}}1_{t\leq\phi_s}=W_t1_{t\leq\phi_s}$.\\
Consider now the event $t >\phi_s$:
\begin{align*}
\mathbb{E}[W_t1_{t>\phi_s}|\mathcal{G}_s]
&= 
W_{s\vee\phi_s}+\mathbb{E}[(W_t-W_{s\vee\phi_s})1_{t> s\vee\phi_s}|\mathcal{G}_s]\\
&= W_{s\vee\phi_s}+\mathbb{E}[\mathbb{E}(W_t-W_{s\vee\phi_s}|\mathcal{G}_s,\phi_s)1_{t> s+\delta_s}|\mathcal{G}_s]1_{t>\phi_s}\\
&= W_{s\vee\phi_s}\\
&= W_{\phi_{\tau(s,t)}}1_{t>\phi_s}
\end{align*}
where we used that
$
\mathbb{E}(W_t-W_{s\vee\phi_s}|\mathcal{G}_s,\phi_s=u)
=
\mathbb{E}(W_t-W_{s\vee u}|\mathcal{G}_s)
=
\mathbb{E}(W_t-W_{s\vee u})
=
0
$
due to the independence of the increments of the Brownian motion.\\
$(iii)$ The integral of the continuous function $u\mapsto W_{\phi_u}$ with respect to the distribution of $\tau(s,t)$ is obtained from $(ii)$ by conditioning with respect to the value of $\tau(s,t)$.
\end{proof}

$\mathbb{P}(\tau(s,t)\in du|\mathcal{G}_s)$ in Lemma \ref{lemma-tau} can be interpreted as the \textit{likelihood} of $\tau(s,t)$ given the path of $(W_u,W_{u+\delta_u})_{u\leq s}$. We obtain the following theorem by combining Lemma \ref{lemma-tau} and Corollary \ref{cor-alpha-from-gaussian-x}

\bigskip

\begin{theorem} \label{cor-tau}
For every $s\geq t$ define the random time $\tau(s,t):=\inf\{0\leq u\leq s:u\vee\phi_u=t\}$ and suppose that it admits a $\mathcal{G}_s$-conditional density $u\mapsto f(u;s,t)$, with respect to Lebesgue measure, which is continuously differentiable in $(s,t)$. Then the $\mathcal{G}$-information drift $\alpha$ of the Brownian motion $W$ is given by
\begin{equation*}
\alpha_s = \int_0^s W_{\phi_u} \partial_t f(u;s,t) \Big|_{t=s} du.
\end{equation*}
\end{theorem}

\paragraph{Application.}

Let us illustrate the use that one can make of the information drift with a simple yet suggestive constrained maximization problem, inspired from an example found in \cite{Shannon} itself revisiting an initial idea from \cite{Kar-Pik}. We consider a market where the the price process is given by a local martingale $M$ (e.g. under a risk neutral measure) for some filtration $\mathcal{F}$.

Together with the characterization of $NA_1$ in \cite{Kar-NA1-ftap} mentioned in Section \ref{sec-exp-fil-id}, Theorem \ref{thm-alpha-lm-deflator} shows that absence of (strict) arbitrage opportunities with the insider information requires 
the existence of the information drift, which defines in that case the predictable component of the dynamics of 
the price process $M$ for the insider. 
We suppose accordingly that $M$ is a $\mathcal{G}$-semimartingale with
decomposition $M=\widetilde{M}+\alpha\cdot [M,M]$ for some process $\alpha\in\mathcal{S}(\mathcal{G})$. 
We attempt here to understand the statistical value of the additional information beyond strict arbitrage by 
expressing quantitatively the statistical advantage of the insider in some simple portfolio optimization problems 
with asymmetric information.

Suppose that, given a filtration $\mby$, initial wealth $x$ and risk aversion $\lambda>0$, every agent tries to optimize his risk-adjusted return according to the program
\begin{align*}
& \sup_{H\in\mathcal{S}(\mby)} \ V^\lambda(x,\mby) := x+\mathbb{E}(H\cdot M)_T - \lambda \mathbb{V}ar(H\cdot M)_T\\
& s.t. \  x+H\cdot M\geq 0 \ on \ [0,T].
\end{align*}
Let $\mathfrak{v}^\lambda(x,\mby)$ be the value of the above problem, and for $H\in\mathcal{S}(\mby)$, let $V(x,H):=\mathbb{E}[(H\cdot M)_T]$ be the corresponding expected return process, and if $H^*$ is an optimal strategy also let $v^\lambda(x,\mby)=V(x,H^*)$.

For an agent with filtration $\mathcal{F}$, $\mathbb{P}$ is a risk-neutral probability and $H^*=0$ is an optimal strategy, with corresponding expected return $V(x,H^*)=x$ and risk-adjusted return $V^\lambda(x,H^*)=x$. However, for an agent with filtration $\mathcal{G}$,
$$\mathbb{E}[(H\cdot{M})_T] -
 \lambda \mathbb{V}ar(H\cdot M)_T
 =\mathbb{E}\big[(H\cdot\widetilde{M})_T+\int_0^T \alpha_s H_s d[M,M]_s - \lambda \int_0^T H_s^2d[M,M]_s\big]$$ 
and is maximal for the $\mathcal{G}$-predictable strategy $H_s^*=\frac{\alpha_s}{2\lambda},s\in I$,
with corresponding expected return $V(x,H^*)=x+\mathbb{E}\int_0^T \alpha_sH_s^*d[M,M]_s=x+\mathbb{E}\int_0^T \frac{\alpha_s^2}{2\lambda}d[M,M]_s$
and risk-adjusted return
$V^\lambda(x,H^*)=x+\frac{1}{2}\mathbb{E}\int_0^T \alpha_sH_s^*d[M,M]_s=x+\mathbb{E}\int_0^T \frac{\alpha^2}{4\lambda}d[M,M]_s$.
This can be summarized as
\begin{equation*}
v^\lambda(x,\mathcal{G})-v^\lambda(x,\mathcal{F})=\mathbb{E}\int_0^T \frac{\alpha_s^2}{2\lambda}d[M,M]_s.
\end{equation*}
This means that in this example a typical dollar value of the additional information is given by $\mathbb{E}\int_0^T \frac{\alpha_s^2}{2\lambda}d[M,M]_s$, which corresponds to the average additional value of the portfolio of a trader with information $\mathcal{G}$ (behaving optimally and with the same risk aversion). The difference between the optimal risk-adjusted returns is also proportional to the $\mathscr{H}^2$ norm of $\alpha$:
\begin{equation*}
\mathfrak{v}^\lambda(x,\mathcal{G})-\mathfrak{v}^\lambda(x,\mathcal{F})=\mathbb{E}\int_0^T \frac{\alpha_s^2}{4\lambda}d[M,M]_s.
\end{equation*}

\section*{Conclusion}

We have highlighted in the introduction that although traditionally associated to insider trading, expansions of filtrations 
can actually be applied to model various complex practical problems in engineering, physics, or finance.
Nevertheless a widespread use of expansion of filtrations in quantitative analysis is limited by important challenges: 
first, it is difficult to design non-degenerate models of expansions of filtrations beyond the initial and progressive expansions
which have been well understood but are often too restrictive to be applied to solve quantitative problems;
second, it is in general challenging to obtain a computable expression for the information drift, even in the these simpler 
models of expansions.

This work following \cite{KP} hopes to bring some progress towards both directions, 
by introducing new examples of dynamical expansions of filtrations, 
and reducing the computation of the information drift associated
to an expansion with a stochastic process to the case of initial expansions.
Although our models do not treat a fully general case, expansions with a stochastic process 
can be extended to a countable number of processes, and this situation would come close to reaching 
a satisfactory level of generality.
We are confident that our hypotheses for convergence of information drifts can be weakened 
using $L^p$ convergence of filtrations (see \cite{LNthesis}).

Other recent work offer other perspectives 
for the computability of the information drift, Malliavin calculus techniques introduced in \cite{IMK-Malliavin}, anticipative stochastic calculus studied in \cite{Oksendal-al} offers hope. They can be transposed in the framework of 
functional It{\^o} calculus following \cite{LNthesis}.
Another perspective lies in the innovative approach adopted in \cite{Ank-thesis,Shannon} 
which relies on techniques of embedding probability spaces inspired by the search of coupling measures solving weakly conditioned SDEs \cite{Baudoin}.

Thus, designing a class of more general expansions with stochastic processes 
for which the information drift can be computed is an interesting and ambitious objective at this point,
towards which this article hopes to bring some progress.
This work hopefully will lead the way to a quantitative method to estimate the information drift, 
and thereby the value associated with expansions of filtrations in new anticipative dynamical examples.

\end{document}